\documentclass[11pt]{amsart}
\usepackage{amssymb,amsmath}
\usepackage{color}
\usepackage[nobysame]{amsrefs}

\newcommand{\nocontentsline}[3]{}
\newcommand{\tocless}[3]{\bgroup\let\addcontentsline=\nocontentsline#1{#2}\egroup}

\newcommand{\dist}{{\rm dist}}

\def\Bbb{\mathbb}
\def\Cal{\mathcal}
\def\Dt{\partial_t}
\def\px{\partial_x}
\def\eb{\varepsilon}

\def\R {\mathbb{R}}

\def\<{\left<}
\def\>{\right>}
\def\Ree{\operatorname{Re}}
\def\Nx{\px}
\def\Dx{\px^2}
\def\({\left(}
\def\){\right)}

\newtheorem{proposition}{Proposition}[section]
\newtheorem{theorem}[proposition]{Theorem}
\newtheorem{corollary}[proposition]{Corollary}
\newtheorem{lemma}[proposition]{Lemma}

\theoremstyle{definition}
\newtheorem{definition}[proposition]{Definition}

\newtheorem{remark}[proposition]{Remark}

\numberwithin{equation}{section}

\def \no#1#2#3 {{\bf #1} (#3), #2.}
\def \eds#1#2#3 {#1, #2, #3.}

\title[Large dispersion] {Large dispersion, averaging and attractors: three 1D paradigms}
\author[A. Kostianko, E. Titi and S. Zelik]{ Anna Kostianko${}^1$, Edriss Titi${}^{2,3}$ and Sergey Zelik${}^1$}
\address{${}^1$
University of Surrey, Department of Mathematics,
Guildford, GU2 7XH, United Kingdom, a.kostianko@surrey.ac.uk, s.zelik@surrey.ac.uk.}
\address{${}^2$ Department of Mathematics, Texas A \& M University, 3368 TAMU, College Station, TX 77843-3368, USA, titi@math.tamu.edu.}
\address{${}^3$ Weizmann Institute of Science,
Department of Computer Science
and Applied Mathematics,
P.O. Box 26,
Rehovot, 76100,
Israel, edriss.titi@weizmann.ac.il.}

\begin{document}
\begin{abstract} The effect of rapid oscillations, related to large dispersion terms, on the dynamics of dissipative evolution equations is studied for the model examples of the 1D complex Ginzburg-Landau and the Kuramoto-Sivashinsky equations. Three different scenarios of this effect are demonstrated. According to the first scenario, the dissipation mechanism is not affected and the diameter of the global attractor remains uniformly bounded with respect to the very large dispersion coefficient. However, the limit equation, as the dispersion parameter tends to infinity, becomes a gradient system. Therefore, adding the large dispersion term actually suppresses the non-trivial dynamics. According to the second scenario, neither the dissipation mechanism, nor the dynamics are essentially affected by the large dispersion and the limit dynamics remains complicated (chaotic). Finally, it is demonstrated in the third scenario that the dissipation mechanism is completely destroyed by the large dispersion, and that the diameter of the global attractor grows together with the growth of the dispersion parameter.
\end{abstract}
\maketitle
\tableofcontents
\section{Introduction}\label{s0}
The study of systems involving rapid oscillations and their averaging is one of the central subjects of the classical theory of dynamical systems which attract a great permanent interest during the last century, see, e.g., \cites{arnold,bo,LM,mit,poi,PS,SVM} and references therein. Roughly speaking, it is well-known that the structure of the averaged equations (as well as their validity) is determined by certain resonance interactions. These resonances typically become very complicated when multi-frequency systems are considered, which makes the corresponding
averaging problem non-trivial and challenging (e.g., due to the presence of small divisors, see, for instance, \cite{arnold} and references therein).
Nevertheless, a lot of averaging results are now available for the case of PDEs (which at least formally contain infinitely many frequencies) in both Hamiltonian and dissipative cases, see, e.g., \cites{BIT,BMN,BLP,bour,bour1,bour2,col1,col2,col3,cv,efz,GuoSimonTiti,Kuksin,lg,tt,ZKO} and references therein.
\par
Very often the analytic structure of the limiting averaged equations is essentially simpler than the structure of the initial problem which allows us to obtain a reasonable description of the initial dynamics involving rapid oscillations using the averaged equations and singular perturbation technique (see, e.g., \cites{W1,W2}). In particular, the presence of rapid oscillations may prevent solutions from blowing up in finite time
(e.g., for the complex Burgers equation with fast rotation, see \cites{BIT,LT}), or may give the global in time regularity (e.g., for the 3D Navier-Stokes equations involving strong Coriolis force term, see \cite{BMN}) or global well-posedness of weak solutions due to the averaging effects in large Fourier modes (e.g., for the KdV equations, see  \cites{BIT, bour1,bour2,GuoSimonTiti}). However, the opposite effects when the presence of rapid oscillations destroys the dissipation mechanism and makes the dynamics essentially more complicated are also known (e.g., for the damped hyperbolic equations, see \cite{zel}). It is remarkable to observe that, even though  the limit averaged equations may remain relatively simple, and dissipative in this case, the global dynamics of the initial system involving rapid oscillations cannot be described by these averaged equations no matter how fast the oscillations are. Notably, it is not always possible to obtain closed system giving the limit for the oscillatory dynamics or it may be even not clear how to split the dependent variables into the "slow" and "fast" parts. In these cases, an approach related with Young measures and the so-called slow observables may help to overcome the problem, see \cites{akst,alt} for the details.
\par
The aim of the present paper is to study
the effect of rapid oscillations induced by large
dispersion on
the long-time dynamics of dissipative PDEs. For simplicity, we restrict ourselves to consider only the 1D
complex Ginzburg-Landau and the Kuramoto-Sivashinsky equations where, on the one hand, the resonances are not complicated and the averaged equations possess a complete description and, on the other hand, a number of non-trivial phenomena which, as believed, have general nature  can be detected.
\par
 To be more
 precise, we consider the following three model problems on $\R$:
 \begin{align}
 \Dt u=(1+iL)\px^2u+\beta
 u-(1+i\omega)u|u|^2;\label{0.eq2}\\
 \Dt u=(1+i\gamma)\px^2u+\beta
 u-(1+i\omega)u|u|^2+L\px^3 u;\label{0.eq1}\\
 \Dt u=-\px^4 u-a\px^2 u+u\px u+L\px^3 u,\label{0.eq3}
 \end{align}
subject to periodic boundary conditions with fundamental periodic domain $(-\pi,\pi)$. Here
$a,\gamma,\omega$ are given real parameters, $\beta$ is a given
complex parameter and $L$ is a large real parameter. In the first two
equations $u$ is assumed to be a complex-valued function:
$u=u_1+iu_2$, and it is real-valued in the third equation.
\par
Based on the analysis given below, we detect three principally different scenarios (paradigms) of how the large dispersion may effect the global dynamics:
\par
{\bf Paradigm I} (corresponds to equation \eqref{0.eq2}). The dissipation mechanism is not affected by large values of the dispersion parameter $L$ and the diameter of the global attractor remains bounded, as $L\to\infty$. However, the large dispersion limit {\it trivializes} the dynamics and the limit averaged equations is a {\it gradient} system, up to some simple change of variables. Thus, the global attractor of the limit equation consists of equilibria and heteroclinic connections only (the so-called regular attractor) and no complicated dynamics is possible in the non-averaged equations, when $L$ is large, see section \ref{s4} for more details.
\par
{\bf Paradigm II} (corresponds to equation \eqref{0.eq1}). Similarly to the previous one, the dissipation mechanism is not affected. However, the complete trivialization of the dynamics, at the infinite dispersion limit, does not happen. Although the global attractor of the limit averaged equations is described by the finite system of ODEs, these ODEs are far from being a gradient system and their dynamics is chaotic (at least for a range of the values of parameters $(\gamma,\beta,\omega)$ that forms an open set in the space of parameters). Then this chaos persists in the initial equations \eqref{0.eq1}, when $L$ is large enough,
see Remark \ref{Rem4.turaev} below, and also \cite{OTZ} for more details.
\par
{\bf Paradigm III} (corresponds to equation \eqref{0.eq3}). The large dispersion destroys the dissipation mechanism and the diameter of the global attractor grows, as $L\to\infty$. In particular, as shown below, see Proposition \ref{Prop4.KSlow}, the $L^2$-norm of the global attractor associated with equation \eqref{0.eq3} grows proportionally to $L$, as $L\to\infty$. Moreover, we provide here rigorous justification to the numerical investigation, reported in \cite{erc}.
\par
The paper is organized as follows.
The resonances and the corresponding infinite dispersion limit  equations are presented in section \ref{s1} for all three cases of equations \eqref{0.eq2}, \eqref{0.eq1} and \eqref{0.eq3}.
The existence of the corresponding global attractors for these limit equations, as well as the upper bounds for the diameter of these global attractors, are verified in section \ref{s2}.
The singular limit, $L\to\infty$, is studied in section \ref{s3}. In particular, the convergence of the global attractors of the Ginzburg-Landau equations to the corresponding global attractors of the infinite dispersion limit  equations, as well as the growing lower bounds for the diameter of the global attractor of the Kuramoto-Sivashinsky equation, are established in section \ref{s3}.
Finally, the gradient structure of the infinite dispersion limit equations, corresponding to the Ginzburg-Landau equations \eqref{0.eq2} and its consequences, are presented in section \ref{s4}.

\section{Preliminaries}\label{s1}
In this section, we introduce
the groups of solution operators associated with the auxiliary linear dispersion
equations:
\begin{align}
\Dt v=L\px^3 v\label{1.gr1},\\
\Dt v=iL\px^2 v\label{1.gr2}
\end{align}
on $\R$, subject to periodic boundary conditions with fundamental periodic domain $(-\pi,\pi)$,
and formulate their simple properties. Moreover, we compute  here
some averages of the non-linearities with respect to the rapid
time oscillations for large values of $L$, generated by these groups, which are crucial for
what follows.
\par
We denote by $H=L^2_{per}(-\pi,\pi)$, the space of
complex-valued $2\pi$-periodic square integrable functions, and
introduce the family $H^s$, $s\in\R$, of Sobolev spaces of periodic
functions with periodic fundamental domain $(-\pi,\pi)$. Let $e_n:=e^{inx}$, $n\in\Bbb Z$, be
the standard orthogonal basis in $H$ and let
$$
v=\sum_{n\in\Bbb Z}v_ne^{inx}.
$$
Neglecting the scalar factor $2\pi$, we define  the $H^s$-norm of $v$  as follows:
\begin{equation}\label{1.norm}
\|v\|_{H^s}^2:=\sum_{n\in\Bbb Z}(|n|^2+1)^s|v_n|^2.
\end{equation}
\begin{lemma}\label{Lem1.iso} The groups of solution operators $\Cal H_L(t)$
and $\Cal F_L(t)$, associated with equations \eqref{1.gr1} and
\eqref{1.gr2} respectively, are  isometries on the Sobolev
spaces $H^s$, for any $s\in\R$. Moreover,  $e_n=e^{inx}$ are
 eigenfunctions satisfying
\begin{equation}\label{1.exp}
\Cal H_L(t)e_n=e^{-iLn^3t}e_n,\ \ \Cal F_L(t)e_n=e^{-iLn^2t}e_n,\ \
n\in\Bbb Z.
\end{equation}
Finally, these groups of solution operators are $2\pi/L$-periodic with respect to time.
\end{lemma}
Indeed, formulas \eqref{1.exp} follow immediately from equations
\eqref{1.gr1} and \eqref{1.gr2} and all other assertions of the
lemma are immediate corollaries of these explicit expressions.
\par
We are going  to change the variable $u$ in equations \eqref{0.eq1} and
\eqref{0.eq3} using the transformation
\begin{equation}\label{1.ch1}
u(t)=\Cal H_L(t)w(t)
\end{equation}
and using the transformation
\begin{equation}
u(t)=\Cal F_L(t)w(t)
\end{equation}
for equation \eqref{0.eq2}. Then we will average the obtained rapidly oscillating in time terms. To this end, we
need to compute the resounant terms appearing from the
nonlinearities. We will do that in the following several lemmas.
We start with the case of equation \eqref{0.eq1}.

\begin{lemma}\label{Lem1.av1} Let $w\in H^s$, $s>1/2$, and let
\begin{equation}\label{1.non1}
F(\tau,w):=\Cal H_1(-\tau)\circ(\Cal H_1(\tau)w\cdot|\Cal H_1(\tau)w|^2).
\end{equation}
Then, for every fixed $\tau$, the map $w\to F(\tau,w)$ is a
bounded smooth map from $H^s$ to itself:
\begin{equation}\label{1.alg}
\|F(\tau,w)\|_{H^s}\le C_s\|w\|_{H^s}^3,
\end{equation}
with a constant $C_s$ that is independent of $\tau$. Moreover, the operator $F$ is
$2\pi$-periodic with respect to $\tau$, and its time averaging has
the form:
$$
\<F(\cdot,w)\>:=1/(2\pi)\int_0^{2\pi}F(\tau,w)\,d\tau=: N(w),
$$
where the operator $N$ has the following explicit form:
\begin{equation}\label{1.avform}
N(w)=2w\|w\|_{H}^2+\bar w [w,w]-2 w_0|w_0|^2e_0-\sum_{n\ne0}w_n(|w_n|^2+2|w_{-n}|^2)e_n.
\end{equation}
with $[w,v]:=\sum_{n\in\Bbb Z}w_nv_{-n}$ and $w=\sum_{n\in\Bbb Z}w_ne_n$. In particular,
\begin{multline}\label{add1}
N(w)=\(2w_0(\|w\|^2_H-|w_0|^2)+\bar w_0[w,w]\)e_0+\\+\sum_{n\ne0}\(w_n(2\|w\|^2_H-|w_n|^2-2|w_{-n}|^2)+\bar w_{-n}[w,w]\)e_n.
\end{multline}
\end{lemma}
\begin{proof} Indeed, estimate \eqref{1.alg} follows immediately
from the fact that $H^s$ is an algebra, for $s>1/2$, and that $\Cal
H_L(\tau)$ are isometries. So, we only need to compute the average of
$F(\tau,w)$. Indeed, inserting
\begin{equation}\label{add2}
u=\Cal H_1(\tau)w=\sum_{n\in\Bbb Z}e^{-in^3\tau}w_n e_n
\end{equation}
to \eqref{1.non1} and using that $e_k e_l=e_{k+l}$, we have
\begin{equation}
F(\tau,w)=\sum_{n,m,k\in\Bbb Z}e^{-i(n^3+m^3+k^3-(n+m+k)^3)\tau}w_n\overline w_{-m}w_k e_{n+m+k}.
\end{equation}
Thus, the resonance condition reads
$$
n^3+m^3+k^3=(m+n+k)^3,
$$
which is equivalent to $(n+m)(m+k)(n+k)=0$. Each of the resonance cases  $n+m=0$ and
$m+k=0$ give the term
$$
\sum_{n\in\Bbb Z}w_n e_n\sum_{m\in\Bbb Z}\bar w_{-m}w_{-m}=w\sum_{m\in\Bbb Z}|w_m|^2=w\|w\|^2_{H}
$$
in the right-hand side of \eqref{1.avform},
 and the case $n+k=0$ gives
the term
$$
\sum_{m\in\Bbb Z}\bar w_{-m}e_m\sum_{n\in\Bbb Z}w_{n}w_{-n}=\bar w[w,w],
$$
where we have used the fact that $\bar e_m=e_{-m}$. However, these three families of resonances  are not disjoint, but
 intersect when $m=n=k=0$, for the zero mode (i.e., the corresponding term $w_0|w_0|^2e_0$ is  counted three times, so the term $2w_0|w_0|^2e_0$ should be substracted). Moreover, there are three pairwise
intersections at $(m,n,k)=(-l,l,l)$, $(l,l,-l)$ and $(l,-l,l)$, for $l\ne0$ (all counted twice). These intersections give the remaining terms in formula
\eqref{1.avform}. It remains to note that \eqref{add1} is equivalent to \eqref{1.avform} and Lemma \ref{Lem1.av1} is proved.
\end{proof}
The next corollary gives the dissipativity of the non-linear
operator $N(w)$.
\begin{lemma}\label{Lem1.dis} The operator $N$, defined by
\eqref{1.avform}, satisfies the following estimate
\begin{equation}\label{1.dis}
[N(w),\bar w]\ge \|w\|^4_H.
\end{equation}
In particular, the value of $[N(w),\bar w]$ is real for every
$w\in H^s$, $s>1/2$.
\end{lemma}
\begin{proof}
Indeed, using \eqref{add1}, we have
$$
[N(w),\bar w]=2\|w\|_{H}^4+|[w,w]|^2-2|w_0|^4-\sum_{n\ne0}|w_n|^2(|w_n|^2+2|w_{-n}|^2),
$$
and the fact that $[N(w),\bar w]$ is real is proved.  Let us prove \eqref{1.dis}. To this end, we transform the last formula to a more convenient  form:
\begin{equation} \label{1.estw}
[N(w),\bar w]=\|w\|^4_H + \sum_{n\in\Bbb Z}|w_n|^2\sum_{m\ne n}|w_m|^2 + \sum_{n\in\Bbb Z}w_n w_{-n}\sum_{m\ne \pm  n}\overline{w_m}\overline{w_{-m}}.
\end{equation}
Using now the Young inequality $|ab|\le\frac12(|a|^2+|b|^2)$, we get
\begin{multline}
|\sum_{n\in\Bbb Z}w_n w_{-n}\sum_{m\ne \pm  n}\overline{w_m}\overline{w_{-m}}|
\le \sum_{n\in\Bbb Z}|w_n||w_{-n}|\sum_{m\ne\pm n}|w_m||w_{-m}|\le\\\le\sum_{n\in\Bbb Z}|w_n||w_{-n}|\frac12\sum_{m\ne\pm n}(|w_m|^2+|w_{-m}|^2)= \sum_{n\in\Bbb Z}|w_n||w_{-n}|\sum_{m\ne\pm n}|w_m|^2\le\\\le\frac12\sum_{n\in\Bbb Z}(|w_n|^2+|w_{-n}|^2)\sum_{m\ne\pm n}|w_n|^2=
\sum_{n\in\Bbb Z}|w_n|^2\sum_{m\ne \pm n}|w_m|^2.
\end{multline}
This estimate together with \eqref{1.estw} gives the desired estimate \eqref{1.dis} and finishes the proof of the lemma.
\end{proof}
\begin{remark} We recall that, according to our notations, $H=L^2_{per}((-\pi,\pi),\Bbb C)$,
$$
\|w\|^2_H=\sum_{n\in\Bbb Z}|w_n|^2=\frac1{2\pi}\int_{-\pi}^\pi|w(x)|^2\,dx
$$
and
$$
(v,w)_H=\sum_{n\in\Bbb Z}v_n\bar w_n=[v,\bar w]=\frac1{2\pi}\int_{-\pi}^\pi v(x)\bar w(x)\,dx
$$
(note that $\bar w=\overline{\sum_{n\in\Bbb Z}w_ne_n}=\sum_{n\in\Bbb Z}\bar w_{-n}e_n$). Thus, inequality \eqref{1.dis} can be rewritten in the form
$$
\operatorname{Re}(N(w),w)_H=(N(w),w)_H\ge\|w\|^4_{H}
$$
which indeed a standard form of the dissipativity condition.
\end{remark}

We now formulate the analogue of Lemma \ref{Lem1.av1} for equation
\eqref{0.eq2}.

\begin{lemma}\label{Lem1.av2} Let $w\in H^s$, $s>1/2$, and let
\begin{equation}\label{1.non2}
G(\tau,w):=\Cal F_1(-\tau)\circ(\Cal F_1(\tau)w\cdot|\Cal F_1(\tau)w|^2).
\end{equation}
Then, for every fixed $\tau$, the map $w\to G(\tau,w)$ is a
bounded smooth map from $H^s$ to itself, and its norms are
uniformly bounded with respect to $\tau$:
\begin{equation}\label{1.alg1}
\|G(\tau,w)\|_{H_s}\le C_s\|w\|_{H^s}^3,
\end{equation}
with a constant $C_s$ that is independent of $\tau$. Moreover, the operator $G$ is
$2\pi$-periodic with respect to $\tau$, and its time averaging has
the form:
\begin{equation}\label{1.avform1}
M(w):=\<G(\cdot,w)\>=2w\|w\|_{H}^2-\sum_{n\in\Bbb Z}w_n|w_n|^2e_n=\sum_{n\in\Bbb Z}w_n(2\|w\|^2_H-|w_n|^2)e_n.
\end{equation}
\end{lemma}
\begin{proof} As before, we only need to compute the average of
$G(\tau,w)$. Inserting
$$
u=\mathcal F_1(\tau)w=\sum_{n\in\Bbb Z}e^{-in^2\tau}w_n e_n
$$
into \eqref{1.non2}, after some elementary calculations we get
\begin{equation}
G(\tau,w)=\sum_{n,m,k\in\Bbb Z}e^{i(n^2-m^2+k^2-(n-m+k)^2)\tau}w_n\overline w_{m}w_k e_{n-m+k}.
\end{equation}
Thus, the resonance condition reads
$$
n^2-m^2+k^2=(n-m+k)^2,
$$
which is equivalent to $(n-m)(k-m)=0$. Thus, we have two families of resonances $n=m$ and $k=m$. Each of them gives the term
$$
\sum_{k\in\Bbb Z}w_ke_k\sum_{n\in\Bbb Z}w_n\bar w_n=w\|w\|^2_H.
$$
Observe that  these resonance families are not disjoint, and intersect when $(n,k,m)=(l,l,l)$, $l\in\Bbb Z$. The common resonance terms are counted twice, so we need to subtract the corresponding term $\sum_{n\in \Bbb{Z}} w_n|w_n|^2e_n$. This gives the desired
 formula \eqref{1.avform1} and finishes the proof of Lemma \ref{Lem1.av2}.
\end{proof}
Analogously to the case of equation \eqref{0.eq1}, we also have
the dissipativity for $M$.
\begin{lemma}\label{Lem1.dis1} The operator $M$, defined by
\eqref{1.avform1}, satisfies
\begin{equation}\label{1.dis1}
[M(w),\bar w]\ge \|w\|^4_H
\end{equation}
and again $[M(w),\bar w]$ is real.
\end{lemma}
\begin{proof} Indeed, according to \eqref{1.avform1},
$$
[M(w),\bar w)]=2\|w\|^4_H-\sum_{n\in\Bbb Z}|w_n|^4=\|w\|^4_H+\sum_{n\in\Bbb Z}|w_n|^2\sum_{m\ne n}|w_m|^2\ge \|w\|^4_H,
$$
and the lemma is proved
\end{proof}

We conclude this section by considering the nonlinearity $u\px u$
associated with the Kuramoto-Sivashinsky equation.
\begin{lemma}\label{Lem1.av3} Let $w\in H^s$, $s>1/2$, and let
\begin{equation}\label{1.non3}
\widetilde H(\tau,w):=\Cal H_1(-\tau)\cdot(\Cal H_1(\tau)w \cdot \px\Cal H_1(\tau)w).
\end{equation}
Then, the operator $\widetilde H(\tau,\cdot)$ is well-defined and is smooth as
an operator from $H^s$ to $H^{s-1}$, and the analogue of uniform
bounds \eqref{1.alg} holds, namely
\begin{equation}\label{1.alg2}
\|\widetilde H(\tau,w)\|_{H^{s-1}}\le C_s\|w\|^2_{H^s},
\end{equation}
with a constant $C_s$ that is independent of $\tau$.
 Moreover, this operator is
$2\pi$-periodic with respect to time, and its time averaging has
the form:
\begin{equation}\label{1.av3}
K(w):=\<\widetilde H(\cdot,w)\>=\<w\>_{sp}\px w+ie_0\sum_{n\in\Bbb
Z}nw_nw_{-n}=w_0\sum_{n\in\Bbb Z}in w_n e_n+\(\sum_{n\in\Bbb Z}inw_nw_{-n}\)e_0,
\end{equation}
where $\<w\>_{sp}=w_0=1/(2\pi)\int_{-\pi}^{\pi}w(x)\,dx$.
\end{lemma}
\begin{proof} As before, we only need to check formula
\eqref{1.av3}. Indeed, inserting \eqref{add2} into \eqref{1.non3}, after some elementary calculations, we get
\begin{equation}\label{1.addksav}
\widetilde H(\tau,w)=\sum_{n,m\in\Bbb Z}in e^{-i(n^3+m^3-(m+n)^3)\tau}w_nw_m,
e_{n+m}
\end{equation}
and, therefore, the resonance condition is
$$
n^3+m^3=(n+m)^3,
$$
which gives $nm(n+m)=0$. Moreover, the case $n=0$ gives nothing due to
the multiplier $in$. So, we only have the cases $m=0$ and $m+n=0$ which give
the first and the second terms in formula \eqref{1.av3} respectively. In contrast to the previous cases, these two families intersect only by $(n,m)=(0,0)$, which gives zero effect on $K$ due to the multiplier $in$. Thus, Lemma \ref{Lem1.av3} is proved.
\end{proof}
\begin{remark}\label{Rem1.real} In the case of real-valued
functions $u$ and $w$ (which is the case for the Kuramoto-Sivashinky equation), we have the additional condition
\begin{equation}\label{1.real}
w_{-n}=\bar w_{n},\ \ n\in\Bbb Z,
\end{equation}
and, therefore, the second term in the expression for $K(w)$
equals zero identically. Moreover, in the case of the Kuramoto-Sivashinsky equation, we
have the additional restriction $w_0=0$.
Thus, in that case the average of $H$ equals zero identically:
\begin{equation}\label{1.av3re}
K(w)\equiv0.
\end{equation}
\end{remark}

\section{Global attractors}\label{s2}
The aim of this section is to formulate and prove some uniform
(with respect to $L$) estimates for the global attractors of
Ginzburg-Landau equations and verify that their analogue does not
take place for the KS equation.   The estimates provided here are formal and can be justified in a rigorous way, for instance, by using Galerkin approximation method and then passing to the limit using the appropriate Aubin compactness theorems, see, e.g., \cites{bv1,cv,tem} and references therein. We start with the case of equation
\eqref{0.eq1}.

\begin{theorem}\label{Th2.attr1} Equation \eqref{0.eq1} is
well-posed in every  $H^s$, with $s\ge0$, and the
following dissipative estimate holds:
\begin{equation}\label{2.dis1}
\|u(t)\|_{H^s}\le Q_s(\|u_0\|_{H^s})e^{-\gamma t}+C_*,
\end{equation}
where the monotone function $Q$ and the positive constants $\gamma$
and $C_*$ depend on $s$, but are independent of $L$, as $L\to\infty$.
Moreover, the following smoothing property is valid:
\begin{equation}\label{2.sm}
\|u(t)\|_{H^s}\le (1+t^{-N_s})\tilde{Q}_s(\|u_0\|_H),\ \ t>0,
\end{equation}
where the monotone function $\tilde{Q}_s$ and the constant $N_s$ are also uniform with respect to $L$, as
$L\to\infty$. Finally, for any two solutions $u_1(t)$ and $u_2(t)$ of problem \eqref{0.eq1}, the following estimate holds:
\begin{equation}\label{add-lip}
\|u_1(t)-u_2(t)\|_{H}\le e^{Kt}\|u_1(0)-u_2(0)\|_{H},
\end{equation}
where the constant $K$ depends on the $H$-norms of $u_i(0)$, $i=1,2$, but is independent of $L$ and $t$.
\end{theorem}
\begin{proof} Since the assertion of the theorem is more or less standard, we give below only brief derivation of the estimates stated in the theorem, see, e.g., \cite{tem} for more details. Moreover, to avoid the technicalities, we derive the dissipative estimate \eqref{2.dis1} for
$s=0$. The estimate in a general case, $s\ge0$, can be obtained in a straightforward way by using the bootstrapping arguments.
\par
 Taking  the inner product in $H$ of equation \eqref{0.eq1} with $u$ and integrating by parts, after
the straightforward transformations, we have
\begin{equation}\label{2.dif0}
\frac12\frac{d}{dt}\|u(t)\|^2_H+\|\Nx u(t)\|^2_{H}+\frac1{2\pi}\|u(t)\|^4_{L^4}\le
\Ree\beta\|u(t)\|^2_H.
\end{equation}
Using that $\frac1{2\pi}\|u(t)\|^4_{L^4}\ge (\Ree\beta+1)\|u(t)\|^2_H-C$,
and applying the Gronwall's inequality to estimate \eqref{2.dif0},
we obtain the following uniform with respect to $L$, as $L\to\infty$,
estimate
\begin{equation}\label{2.dis0}
\|u(t)\|_{H}^2+\int_t^{t+1}\|\Nx u(t)\|_H^2\,dt\le
Q(\|u_0\|_{H})e^{-\gamma t}+C_*,
\end{equation}
which coincides with the desired estimate \eqref{2.dis1}, for $s=0$.
\par
Let us now verify the smoothing property \eqref{2.sm}.
For simplicity, we  deduce estimate \eqref{2.sm} for $s=1$ only (for $s>1$,
it can be obtained in a standard way using  bootstrap
arguments). Indeed, due to the embedding theorem
\begin{equation}\label{add-emb}
L^\infty((t,t+1),L^2((-\pi,\pi))\cap L^2((t,t+1),H^1((-\pi,\pi))\subset
L^6((t,t+1)\times(-\pi,\pi)),
\end{equation}
estimates \eqref{2.dif0} and  \eqref{2.dis0} one can establish an estimate for the
$L^2((t,t+1)\times(-\pi,\pi))$-norm of the nonlinearity. Multiplying now equation
\eqref{0.eq1} by $-t\overline{\partial_x^2u}$ and taking the real part, after the standard
transformations, we  have
\begin{equation}\label{2.dism}
\frac{d}{dt}(t\|\Nx u(t)\|_H^2)+Kt\|\Dx u(t)\|^2_H\le (1+Bt)\|\Nx
u(t)\|^2_H+Ct\|u(t)\|^6_{L^6},
\end{equation}
where $K, B$ and $C$ are positive constants. Integrating this relation with respect to time on $[0,t)$, with $t\in (0,1)$, and using
\eqref{2.dis0}, we deduce estimate \eqref{2.sm} for $s=1$.
\par
Let us now verify the Lipschitz continuity \eqref{add-lip}. Indeed, let $u_1(t)$ and $u_2(t)$ be two solutions of equation \eqref{0.eq1} and let
$v=u_1-u_2$. Then, the function $v$ satisfies
\begin{equation}
\Dt v=(1+i\gamma)\Dx v+\beta v-(1+i\omega)[u_1|u_1|^2-u_2|u_2|^2]+L\partial_x^3v,\ \ v(0)=u_1(0)-u_2(0).
\end{equation}
Taking the inner product in $H$ of this equation with $v$ and arguing as before, we get
$$
\frac12\frac d{dt}\|v\|^2_H+\|\partial_x v\|^2_H-\Ree\beta\|v\|^2_H\le \sqrt{1+w^2}(|u_1|u_1|^2-u_2|u_2|^2,|u_1-u_2|).
$$
The right-hand side of this inequality can be estimated using the H\"older inequality with exponents $3$ and $3/2$, and the interpolation inequality
$\|v\|_{L^3}\le C\|v\|_{L^2}^{5/6}\|v\|_{H^1}^{1/6}$:
\begin{multline}
(|u_1|u_1|^2-u_2|u_2|^2|,|u_1-u_2|)\le C(|u_1|^2+|u_2|^2,|v|^2)\le C(\|u_1\|^2_{L^6}+\|u_2\|^2_{L^6})\|v\|^2_{L^3}\le \\\le C(\|u_1\|^2_{L^6}+\|u_2\|^2_{L^6})\|v\|^{5/3}_H\|v\|_{H^1}^{1/3}\le C(1+\|u_1\|^6_{L^6}+\|u_2\|^6_{L^6})\|v\|^2_{H}+\|\partial_x v\|^2_H.
\end{multline}
Thus, we derived the following estimate:
$$
\frac12\frac d{dt}\|v\|^2_H \le  C(1+\|u_1\|^6_{L^6}+\|u_2\|^6_{L^6})\|v\|^2_{H},
$$
where the constant $C$ is independent of $L$, $u_1$ and $u_2$. Applying the Gronwall inequality to this relation and using that the space-time $L^6$-norm of $u_i$ is under control, due to estimate \eqref{2.dis0} and the embedding \eqref{add-emb}, we derive the desired estimate \eqref{add-lip}.
The uniqueness of a solution follows from \eqref{add-lip}. Thus,
Theorem \ref{Th2.attr1} is proved.
\end{proof}
Let us consider now the Ginzburg-Landau equation \eqref{0.eq2}.
\begin{theorem}\label{Th2.attr2}
 Equation \eqref{0.eq1} is well-posed in every space $H^s$, with $s\ge0$, and
estimates \eqref{2.dis1}, \eqref{2.sm} and \eqref{add-lip}  hold uniformly with
respect to $L$, as $L\to\infty$.
\end{theorem}
The proof of this theorem follows word by word the proof of
the previous one and, thus, is  omitted.
\par
Due to Theorems \ref{Th2.attr1} and \ref{Th2.attr2}, the solution semigroups $S^L_{GL1}(t)$ and $S^L_{GL2}(t)$ associated with equations \eqref{0.eq1} and \eqref{0.eq2} are well-defined  in $H$:
\begin{equation}\label{add-sem}
S_{GLi}^L(t)u(0):=u(t),\ \ t\ge0,\ \ u(0)\in H,
\end{equation}
where $i=1,2$, and $u(t)$ solves  equations \eqref{0.eq1} or \eqref{0.eq2} if $i=1$ or $i=2$ respectively. Moreover, according to estimate \eqref{2.dis1}, these semigroups are dissipative in $H^s$, $s\ge0$:
\begin{equation}\label{add-dis}
\|S^L_{GLi}(t)u_0\|_{H^s}\le Q_s(\|u_0\|_{H^s})e^{-\gamma_s t}+C_s,
\end{equation}
where $Q_s$, $\gamma_s$ and $C_s$ depend on $s$, but are independent of $L$.
\par
Our next step is to study the global attractors of the introduced semigroups.
For the convenience of the reader we start by recalling the definitions related with global attractors, see, e. g., \cites{bv1,tem} for more details.
\begin{definition} Let $X$ be a Banach space and $S(t):X\mapsto X$, $t\ge0$, be a semigroup in $X$. Then, a set $\mathcal B$ is an {\it absorbing} set for the semigroup $S(t)$ if for any bounded subset $B\subset X$ there exists $T=T(B)$ such that
$$
S(t)B\subset\mathcal B
$$
for all $t\ge T$.
\par
A set $\mathcal B\subset X$ is an {\it attracting} set for the semigroup $S(t)$ if for every bounded $B\subset X$ and every open neighbourhood $\mathcal O(\mathcal B)$ there exists time $T=T(B,\mathcal O)$ such that
$$
S(t)B\subset \mathcal O(\mathcal B)
$$
for all $t\ge T$. The attraction property can be rewritten in the equivalent form using the so-called non-symmetric Hausdorff distance.
Namely, $\mathcal B$ is an attracting set for the semigroup $S(t)$ if, for every bounded $B\subset X$,
$$
\lim_{t\to\infty}\operatorname{dist}_X(S(t)B,\mathcal B)=0,
$$
where the non-symmetric Hausdorff distance between sets $U$ and $V$ of $X$ is defined as follows:
$$
\operatorname{dist}_X(U,V):=\sup_{x\in U}\inf_{y\in V}\|x-y\|_X.
$$
\par
Finally, a set $\mathcal A$ is a global attractor for the semigroup $S(t)$ if the following conditions are satisfied:
\par
1) The set $\mathcal A$ is compact in $X$;
\par
2) The set $\mathcal A$ is strictly invariant: $S(t)\mathcal A=\mathcal A$ for all $t\ge0$;
\par
3) The set $\mathcal A$ is an attracting set for the semigroup $S(t)$.
\end{definition}

The next corollary gives the existence  of  global attractors and their uniform bounds with respect to $L$, as $L\to\infty$, for equations \eqref{0.eq1} and \eqref{0.eq2}.
\begin{corollary}\label{Cor2.attr} Let $S_{GL1}^L(t):H\mapsto H$ and $S_{GL2}^L(t):H\mapsto H$ be the solution semigroups generated by equations
\eqref{0.eq1} and \eqref{0.eq2}, respectively. Then these semigroups
 possess global
attractors ($\Cal A_{GL1}(L)$ and $\Cal A_{GL2}(L)$, respectively)
in the phase space $H$. Moreover, these global attractors are uniformly
bounded with respect to $L$, as $L\to\infty$, in $H^s$, for every $s\ge0$:
\begin{equation}\label{2.unattr}
\|\Cal A_{GL1}(L)\|_{H^s}+\|\Cal A_{GL2}(L)\|_{H^s}\le C_s,
\end{equation}
where $C_s$ depends on $s$, but is independent of $L$. Finally, the global attractors $\Cal A_{GLi}(L)$, for $i = 1,2$, can be described as follows:
\begin{equation}\label{add-str}
\Cal A_{GLi}(L)=\Cal K_{GLi}(L)\big|_{t=0},
\end{equation}
where $\Cal K_{GLi}(L)\subset C_b(\R,H^s)$, are the sets of all solutions of equation \eqref{0.eq1} (resp. \eqref{0.eq2}) which are defined for all $t\in\R$ and are bounded.
\end{corollary}
\begin{proof}
Indeed, according to the abstract theorem on the existence of a global attractor existence, see, e. g., \cites{bv1,tem},  we need to verify that
\par
1) Operators $S_{GLi}^L(t)$ are continuous in $H$, with respect to the initial data, for every fixed $t$;
\par
2) Semigroups $S_{GLi}(t)$ possess {\it compact} absorbing sets in $H$.
\par
Note that the first assertion is an immediate corollary of estimate \eqref{add-lip}, so we only need to verify the existence of a compact absorbing set. According to estimate \eqref{add-dis}, the set
$$
\mathcal B:=\{u\in H, \|u\|^2_H\le 2C_*\}
$$
is an absorbing set for the semigroups $S^L_{GLi}(t)$, $i=1,2$. However, this set is not compact in $H$.
To overcome this difficulty, we consider the sets
$$
\mathcal B_i:=S_{GLi}^L(1)\mathcal B,\ \ i=1,2.
$$
Then, due to estimate \eqref{2.sm}, the sets $\mathcal B_i$ are uniformly (with respect to $L$, as $L\to\infty$) bounded in $H^s$, for all $s>0$, and, therefore, they are compact in $H$. Thus, all of the assumptions of the abstract global attractor existence theorem are verified and, consequently,
the global attractors $\mathcal A_{GLi}(L)$, $i=1,2$, exist. Since the global attractor is always a subset of an absorbing set, estimate \eqref{2.unattr} follows from the fact that the absorbing sets $\mathcal B_i$ are uniformly bounded, with respect to $L$, in any $H^s$. The description \eqref{add-str} also follows from the above mentioned global attractor existence theorem and the corollary is proved.
\end{proof}
\begin{remark} It is well-known (see, e.g., \cite{DT}) that the global attractors $\Cal A_{GL1}$
and $\Cal A_{GL2}$ are not only smooth, but also analytic (belong
to certain class of Gevrey regularity). Moreover, arguing in a standard way, one can
also show that they are uniformly (with respect to $L\to\infty$)
bounded in the appropriate Gevrey norm.
\end{remark}
We now turn to the case of Kuramoto-Sivashinsky equation, where the
situation is a bit different.
\begin{theorem}\label{Th2.attr3} Equation \eqref{0.eq3} possesses a unique solution
$u(t)\in H_0:=\{u\in H,\ \ \<u\>_{sp}=0\}$
for every $u_0\in H_0$, and the following estimate holds:
\begin{equation}\label{2.dis3}
\|u(t)\|^2_H\le \|u(0)\|^2_He^{-t}+C(L^2+1),
\end{equation}
where the positive constant $C$  is
independent of $L$, but depends on the parameter $a$.
\end{theorem}
\begin{proof} We  only  verify below the    dissipative
estimate  \eqref{2.dis3}. The existence and uniqueness can be checked in a usual way (see, e.g.,  \cite{tem}).
 To this end, we need the following standard lemma (see, e.g., \cites{cees,g94,NST,ot,tem}).
\begin{lemma}\label{Lem2.phi} For every $M>0$, there exists an odd function
$\phi=\phi(M)\in C^\infty\cap H_0$ such that the following inequality
\begin{equation}\label{2.phi}
\|\Dx w\|^2_{H}- ( w \Nx \phi, w)\ge M\|w\|_{H}^2
\end{equation}
holds for all $w\in H^2$ such that $w(0)=0$.
\end{lemma}

Following \cite{g94}, we  fix $\phi$ from Lemma \ref{Lem2.phi} with $M$ being large enough, consider a family  of shifted functions
$\phi_s(x):=\phi(x-s)$, $s\in\R$, and
introduce a functional
\begin{equation}\label{add-min}
F(u(t)):=\min_{s\in[-\pi,\pi]}\|u(t)-\phi_s\|_{H}^2,
\end{equation}
where $u$ is a solution of equation \eqref{0.eq3}. Obviously, the
minimum exists and is achieved for some value $s^*=s(t)$ which satisfies the
following orthogonality condition:
\begin{equation}\label{2.ort}
(u(t)-\phi_{s^*},\Nx \phi_{s^*})=0.
\end{equation}
Let now $v(t):=u(t)-\phi_{s^*(t)}$. Then, this function {\it formally} satisfies
the equation
\begin{equation}\label{2.eqmod}
\Dt u=\Dt v+\Dt s^*(t)\Nx \phi_{s^*}=-\Nx^4v-a\Dx v+\phi_{s^*}\Nx v+\Nx \phi_{s^*}v+v\Nx v+L\Nx^3 v+f_L(t),
\end{equation}
where $f_L(t):=-\Nx^4\phi_{s^*}-a\Dx
\phi_{s^*}+\phi_{s^*}\Nx\phi_{s^*}+L\Nx^3\phi_{s^*}$. Multiplying
this equation, again formally, by $v$, integrating by $x$ and using the
orthogonality condition \eqref{2.ort}, we obtain
\begin{equation}\label{2.est1}
\frac{1}{2} \frac{d}{dt} \|v\|_{H}^2+\|\Dx v\|^2_{H}- \frac{1}{2}(v \Nx \phi_{s^*},v)=a\|\Nx
v\|^2_{H}+(f_L,v).
\end{equation}
In addition, we claim that the inequality
\eqref{2.phi} with $\phi=\phi_{s^*(t)}$ is satisfied for $w=v(t)$, for every $t$
(although $v(t,s^*)\ne0$ in general). Indeed, from Lemma
\ref{Lem2.phi}, we know that \eqref{2.phi} with $\phi=\phi_{s^*}$
holds for any $w\in H^2$ such that $w(s^*)=0$.
Let now  $K:=v(t,s^*)$ and
$w=v-K$. Then, using again the orthogonality condition and the facts that
 $\<v\>_{sp}=\<\Nx \phi_{s^*}\>_{sp}=0$ and $w(s^*)=0$, we get
\begin{multline}\label{2.est2}
\|\Dx v\|^2_H-(\Nx \phi_{s^*}v,v)=\|\Dx w\|^2_{H}-(\Nx \phi_{s^*}w,w)-2K(\Nx \phi_{s^*},v)+
 K^2\<\Nx \phi_{s^*}\>_{sp}=\\= \|\Dx w\|^2_H-(\Nx \phi_{s^*}w,w)\ge
M\|w\|^2_H=\\=M(\|v\|^2_{H}-2 K\<v\>_{sp}+ K^2)=M\|v\|^2_H+
 MK^2\ge M\|v\|^2_H.
 \end{multline}
Thus, estimates \eqref{2.est1} and \eqref{2.est2} give
\par

%
%
\begin{equation}\label{2.est2_odd}
\frac{d}{dt}\|v\|^2_H+\|v\|^2_{H^2}+ M\|v\|^2_H\le
2a\|v\|_{H^1}^2+2|(f_L,v)|.
\end{equation}
It remains to recall that $\|f_L(t)\|_{H}\le C(|L|+1)$ and that, due to interpolation, we have
\begin{equation}\label{add3}
2 a\|v\|^2_{H^1}\le  \|v\|^2_{H^2}+\left( M-2\right)\|v\|_H^2
\end{equation}
if $M$ is large enough, depending on the parameter $a$. This gives
\begin{equation}\label{add-35}
\frac{d}{dt}\|v\|^2_{H}+\|v\|^2_H\le\|f_L\|^2_H\le C(L^2+1),
\end{equation}
and by the Gronwall's inequality we {\it formally} derive the desired estimate \eqref{2.dis3}.

However, there is still an essential gap in the proof, namely, although the minimizer $s^*=s^*(t)$ for  problem \eqref{add-min} exists, it is non-necessarily unique. To overcome this problem, we may select one-valued (measurable) branch of the multi-valued function $s^*(t)$ (which will be also denoted by $s^*(t)$). More essential is the fact that $s^*(t)$ may be non-differentiable and even have discontinuities (jumps) at
 the points where the value $s^*(t)$ is not unique. Thus, neither $\Dt s^*(t)$ nor $\Dt v(t)$ are  properly defined and the multiplication of equation \eqref{2.eqmod} by $v(t)$ should be justified. This is done in the following lemma.

 \begin{lemma}\label{Lem2.just-add} Let $u(t)\in H_0$, $t\in[0,T]$, be smooth. Then the function $t\to F(u(t))$ is absolutely continuous in time and the following formula is valid:
 \begin{equation}\label{add-Fdir}
 \frac d{dt}F(u(t))=\frac d{dt}\|v\|^2_H=2(\Dt u(t),v(t)), \text{ for almost all $t\in[0,T]$,}
 \end{equation}
 where $v(t)=u(t)-\phi_{s^*(t)}$ and $s^*(t)$ is a minimizer of \eqref{add-min}.
 \end{lemma}
\begin{proof} We first note that, due to the triangle inequality, the distance function
 $$
 u\to \operatorname{dist}(u,\{\phi_s,\, s\in[-\pi,\pi]\})
 $$
is Lipschitz continuous with Lipschitz constant one. For this reason, for any $t_1,t_2\in[0,T]$, we have
\begin{equation}\label{add-5}
|\,\|v(t_1)\|_H-\|v(t_2)\|_H|\le\|u(t_1)-u(t_2)\|_{H},
\end{equation}
so the function $t\to\|v(t)\|^2_H$ is Lipschitz continuous despite the fact that $v(t)$ may have jumps. Moreover, if the function $u$ is regular enough, we have
\begin{equation}\label{add-6}
\|u(t_1)-u(t_2)\|_H\le C|t_1-t_2|.
\end{equation}
Inequalities \eqref{add-5} and \eqref{add-6} show that the function $t\to\|v(t)\|^2_H$ is absolutely continuous and, therefore, is differentiable for almost all $t\in[0,T]$. Thus, we only need to find the derivative. To this end, we take $0<t_2<t_1<T$ such that $\|v(t)\|^2_H$ is differentiable at $t_1$ and $t_2$, and we introduce $w=u(t_1)-u(t_2)$. Then using the fact that $F(u(t_i))$ is a minimum, we have
$$
\|u(t_1)-\phi_{s_1}\|_H^2\le \|u(t_1)-\phi_{s_2}\|_H^2=\|u(t_2)-\phi_{s_2}+w\|_H^2=\|u(t_2)-\phi_{s_2}\|_H^2+2(u(t_2)-\phi_{s_2},w)+\|w\|^2_H
$$
and, analogously,
$$
\|u(t_2)-\phi_{s_2}\|^2_H\le \|u(t_1)-\phi_{s_1}\|_H^2-2(u(t_1)-\phi_{s_1},w)+\|w\|_H^2,
$$
where $s_i=s^*(t_i)$. From these two inequalities we conclude that
\begin{multline}\label{add-7}
2\(v(t_1),\frac{u(t_1)-u(t_2)}{t_1-t_2}\)-\frac{\|u(t_1)-u(t_2)\|_H^2}{t_1-t_2}\le \frac{\|v(t_1)\|_H^2-\|v(t_2)\|_H^2}{t_1-t_2}\le\\\le 2\(v(t_2),\frac{u(t_1)-u(t_2)}{t_1-t_2}\)+\frac{\|u(t_1)-u(t_2)\|_H^2}{t_1-t_2}.
\end{multline}
Passing now to the limit $t_1\to t_2$ in the right-hand side of \eqref{add-7}, we see that, for almost all $t_2\in[0,T]$,
$$
\frac d{dt}\|v(t)\|_H^2\big|_{t=t_2}\le 2(\Dt u(t_2),v(t_2))
$$
(recall that the existence of the derivative has already been proved). Analogously, passing to the limit $t_2\to t_1$ in the left-hand side of inequality
\eqref{add-7}, we get
$$
\frac d{dt}\|v(t)\|_H^2\big|_{t=t_1}\ge 2(\Dt u(t_1),v(t_1))
$$
for almost all $t_1\in[0,T]$. The last two inequalities imply \eqref{add-Fdir} and finish the proof of the lemma.
\end{proof}
Now, it is not difficult to finish the proof of the theorem. Indeed, the already proved lemma justifies the derivation of the key inequality \eqref{add-35} for the case where the solution $u(t)$ is smooth (which will be the case if we start from smooth initial data). The validity of the desired estimate \eqref{2.dis3} for any $u(0)\in H_0$ follows then by the usual approximation arguments.
Thus, Theorem \ref{Th2.attr3} is proved.
\end{proof}
\begin{corollary}\label{Cor2.attr3} Let the assumptions of Theorem \eqref{Th2.attr3} hold. Then, the Kuramoto-Sivashinsky equation \eqref{0.eq3}
possesses a global attractor $\Cal A_{KS}(L)$, in the phase space
$H_0$. Moreover, this global attractor is smooth and
\begin{equation}\label{2.KSatt}
\|\Cal A_{KS}(L)\|_{H^s}\le C_s(L),
\end{equation}
for some positive constant $C$ dependent on $s$, $a$ and $L$. In
particular, $C_0(L)=C L$ with some $C$ independent of $L$, but dependent on the parameter $a$.
\end{corollary}
\begin{proof}
Indeed, analogously to \eqref{2.sm}, we have
the smoothing property on a finite interval  for the parabolic  equation  \eqref{0.eq3} (see,
e.g., \cite{he}):
\begin{equation}\label{2.estsmooth}
\|u(t+1)\|_{H^s}\le C(1+\|u(t)\|_H^{N_s})
\end{equation}
for some positive $C$ and $N_s$. Moreover, the analogue of the Lipschitz estimate \eqref{add-lip} also holds for this equation (and can be proved analogously to the proof given in Theorem \ref{Th2.attr1}). Thus, the proof of the corollary repeats word by word the proof of Corollary \ref{Cor2.attr} and, for this reason, is omitted.
\end{proof}
\begin{remark} As before, one can show that the global attractor $\Cal
A_{KS}$ is analytic in the sense of Gevrey class. However, in
contrast to the cases of Ginzburg-Landau equations, its Gevrey
norm is not uniformly bounded as  $L\to\infty$. As we will see in the next section, the norms of the global attractors $\mathcal A_{KS}(L)$ indeed grow, as $L\to\infty$.
\end{remark}

\section{The limit $L\to\infty$}\label{s3}

The main aim of this section is to study the dependence of the
above constructed global attractors on $L$, as  $L\to\infty$. We start with
the case of complex Ginzburg-Landau equations \eqref{0.eq2} and
\eqref{0.eq1}. To this end, we change the dependent variable $u$
as follows
\begin{equation}\label{3.change}
\begin{cases}
u(t)=\Cal H_L(t)w(t)\ \ \text{for the case of equation
\eqref{0.eq1}},\\
u(t)=\Cal F_L(t)w(t)\ \ \text{for the case of equation
\eqref{0.eq2}},
\end{cases}
\end{equation}
where the isometries $\Cal H_L$ and $\Cal F_L$ are defined by
\eqref{1.exp}. Then, introducing a small parameter $\eb:=\frac1L$,
we rewrite equations \eqref{0.eq1} and \eqref{0.eq2} in the form
\begin{equation}\label{3.eq1}
\Dt w=(1+i\gamma)\px^2w+\beta
 w-(1+i\omega)F(t/\eb,w),
\end{equation}
and
\begin{equation} \label{3.eq2}
 \Dt w=\px^2w+\beta
 w-(1+i\omega)G(t/\eb,w)
\end{equation}
respectively. The functions $F$ and $G$ are defined by \eqref{1.non1}
and \eqref{1.non2}.
\par
Equations \eqref{3.eq1} and \eqref{3.eq2} contain rapidly
oscillating in time terms $F(t/\eb,w)$ and $G(t/\eb,w)$ and we can
(formally) write the limit averaged equations:
\begin{equation}\label{3.aveq1}
\Dt \widehat w=(1+i\gamma)\px^2\widehat w+\beta
 \widehat w-(1+i\omega)N(\widehat w),
\end{equation}
and
\begin{equation}\label{3.aveq2}
 \Dt \widehat w=\px^2\widehat w+\beta
 \widehat w-(1+i\omega)M(\widehat w),
\end{equation}
where the nonlocal operators $N$ and $M$ are introduced  in Lemma
\ref{Lem1.av1} and Lemma \ref{Lem1.av2}. Our first assertion gives the
dissipative estimate for the solutions of the limit averaged
equations.
\begin{proposition}\label{Prop3.dis} Problems
 \eqref{3.aveq1} and \eqref{3.aveq2} are globally well-posed for any $\hat w(0)\in H^s$, $s\ge0$, and the corresponding solutions
 $\hat w(t)$ satisfy the analogues of
estimates \eqref{2.dis1}, \eqref{2.sm} and \eqref{add-lip}.
\end{proposition}
Indeed, due to the dissipativity conditions \eqref{1.dis} and
\eqref{1.dis1}, the proof of these assertions repeats word by word
the proof of Theorem \ref{Th2.attr1} and for this reason it is omitted.
\par
Since the averaged problems \eqref{3.aveq1} and \eqref{3.aveq2} are globally well-posed, the corresponding solution semigroups $\hat S_{GL1}(t):H\mapsto H$ and $\hat S_{GL2}(t):H\mapsto H$ are well defined:
$$
\hat S_{GLi}(t)\hat w(0):=\hat w(t),\ \ i=1,2,
$$
where $\hat w(t)$ solves problem \eqref{3.aveq1} if $i=1$, and problem \eqref{3.aveq2} if $i=2$.

\begin{proposition}\label{Prop3.avattr} The solution semigroups associated with averaged equations \eqref{3.aveq1} and \eqref{3.aveq2} possess global attractors $\hat {\mathcal A}_{GL1}$ and $\hat{\mathcal A}_{GL2}$, respectively, which are bounded in $H^s$, for any $s$. Moreover, these semigroups are invariant with respect to $\Cal H_1(s)$ and $\Cal F_1(s)$, respectively:
\begin{equation}\label{add-sem1}
\Cal H_1(s)\circ\hat S_{GL1}(t)=\hat S_{GL1}(t)\circ\Cal H_1(s),\ \ \Cal F_1(s)\circ\hat S_{GL2}(t)=\hat S_{GL2}(t)\circ\Cal F_1(s),
\end{equation}
for all $s\in\R$ and $t\ge0$, therefore,
\begin{equation}\label{3.sym}
\Cal H_L(s)\widehat{\Cal A}_{GL1}=\widehat{\Cal A}_{GL1},\ \
\Cal F_L(s)\widehat{\Cal A}_{GL2}=\widehat{\Cal A}_{GL2},
\end{equation}
for all $L>0$ and $s\in\R$.
\end{proposition}
\begin{proof} Indeed, the existence of global attractors can be verified exactly as in Corollary \ref{Cor2.attr}, and \eqref{3.sym} is an immediate corollary of the invariance \eqref{add-sem1}. Thus, we only need to check \eqref{add-sem1}. In turn, in order to check \eqref{add-sem1}, it is enough to verify the invariance of the nonlinearities $N$ and $M$:
$$
\Cal H_1(s)N(w)=N(\Cal H_1(s)w),\ \ \Cal F_1(s)M(w)=M(\Cal F_1(s)w), \ s\in\R,\ \ w\in H.
$$
Finally, the invariance of $N$ and $M$ can be easily verified using the explicit formulas \eqref{1.exp}, \eqref{add1} and \eqref{1.avform1}. Thus, the proposition is proved.
\end{proof}
Our next observation shows that the global attractors  $\widehat{\Cal
A}_{GL1}$ and $\widehat{\Cal A}_{GL2}$ of the limit equations belong to the finite-dimen\-si\-onal
invariant hyperplane of the phase space and, thus, can be obtained
by solving a system of ODEs.
\begin{proposition}\label{Prop3.inv} Let $H_D$, where $D\in\Bbb N$ be the
$2D+1$-dimensional hyperplane
\begin{equation}
H_D:=\bigg\{\sum_{n=-D}^Da_ne_n,\ \ a_n\in\Bbb C\bigg\}.
\end{equation}
Then the hyperplanes $H_D$ are invariant with respect to equations \eqref{3.aveq1} and
\eqref{3.aveq2}, for all $D$. Moreover, their global attractors belong to $H_D$ provided $D=D(\beta,\omega)$ is large enough:
\begin{equation}\label{3.finattr}
\widehat{\Cal A}_{GL1}\subset H_D, \ \ \widehat{\Cal A}_{GL2}\subset H_D.
\end{equation}
In particular, we may take $D=[\sqrt{\Ree\beta}]+1$ for the case of equation \eqref{3.aveq2}, here $[z]$ denotes the integer part of the number $z$ (the global attractor of \eqref{3.aveq2} is trivial if $\Ree\beta<0$).
\end{proposition}
\begin{proof} Indeed, the invariance of hyperplanes $H_D$ follows immediately from the explicit
structure of nonlinearities $N(w)$ and $M(w)$ given by \eqref{add1} and \eqref{1.avform1}. So we only need to prove the embeddings \eqref{3.finattr}
if $D$ is large enough. Let us start with the case of the equation \eqref{3.aveq1}.
\par
Let us consider the equations for $w_k$ and
$w_{-k}$ from \eqref{3.aveq1}. Then,
multiplying them by $\bar w_k$ and $\bar w_{-k}$ respectively, and taking a sum and the real part, after the standard calculations, we get
\begin{multline}\label{3.add1}
\frac12\frac d{dt}(|w_k|^2+|w_{-k}|^2)+(k^2-\Ree\beta)(|w_k|^2+|w_{-k}|^2)+\\+\Ree((1+i\omega)([N(w)]_k\cdot
\bar w_k+[N(w)]_{-k}\cdot \bar w_{-k}))=0,
\end{multline}
where we denote by $[N(w)]_k$ the $k$th coordinate of $N(w)$ in the basis $\{e_n\}_{n\in\Bbb Z}$. Using the explicit formula \eqref{add1} for the coordinates of $N(w)$, we see that
$$
|\Ree((1+i\omega)([N(w)]_k\cdot
\bar w_k+[N(w)]_{-k}\cdot \bar w_{-k}))|\le C\sqrt{1+\omega^2}\|w\|^2_H(|w_k|^2+|w_{-k}|^2),
$$
where the constant $C$ is independent of $k$ and $\omega$. Thus, \eqref{3.add1} reads
\begin{equation}\label{3.add11}
\frac12\frac d{dt}(|w_k|^2+|w_{-k}|^2)+(k^2-\Ree\beta-C\sqrt{1+\omega^2}\|w\|^2_H)(|w_k|^2+|w_{-k}|^2)\le 0.
\end{equation}
Since the $H$-norm of $w$ is bounded on the global attractor, for sufficiently large $k$, the second term in \eqref{3.add11} becomes positive and the Gronwall's inequality gives
$$
|w_k(t)|^2+|w_{-k}(t)|^2\le
e^{-\alpha t}(|w_k(0)|^2+|w_{-k}(0)|^2), \ \text{for some}\  \alpha>0.
$$
Thus, $w_k(t)=w_{-k}(t)=0$ on the global attractor, for $|k|$ large enough, and embedding \eqref{3.finattr} is verified for the case of equation \eqref{3.aveq1}.
\par
Let us now consider equation \eqref{3.aveq2}. The situation here is simpler since the expression
$$
[M(w)]_k\bar w_k=|w_k|^2(2\|w\|^2_H-|w_k|^2)
$$
is real and is non-negative and the analogue of \eqref{3.add11} reads
$$
\frac12\frac d{dt}|w_k|^2+(k^2-\Ree \beta)|w_k|^2\le0.
$$
Thus, indeed, $w_k(t)=0$ on the global attractor if $k^2>\Ree\beta$, and the proposition is proved.
\end{proof}


 The next result shows that
the distance between the appropriate averaged and non-averaged
trajectories is indeed small on the finite time interval.
\begin{theorem}\label{Th3.average} Let $w_0\in H^s$, for some
$s\ge1$, and let $w_\eb(t)$ and $\widehat w(t)$ be the non-averaged and
averaged solutions of \eqref{3.eq1} and \eqref{3.aveq1}, respectively, (or of
\eqref{3.eq2} and \eqref{3.aveq2}, respectively). Assume also that
$$
w_\eb(0)=\hat w(0)\in H^s.
$$
Then, the following estimate
holds:
\begin{equation}\label{3.dist}
\|w_\eb(t)-\widehat w(t)\|_{H^s}\le C_s\eb e^{Kt},
\end{equation}
where the constants $C_s$ and $K$ depend on $s$ and on the $H^s$ norm of $w(0)$, but are
independent of $\varepsilon$, as $\eb\to0$.
\end{theorem}
\begin{proof} Estimate of the form \eqref{3.dist} is a standard
result of the averaging theory and can be referred to as the first
Bogolyubov theorem, see, e.g., \cites{bo,cv,lg,SVM}. However, verifying that the difference between the averaged and non-averaged solutions is of order $O(\eb)$ requires some analysis especially in the case of PDEs. For this reason,
we sketch below  the proof of this fact for the case
of equation \eqref{3.eq1} and $s=1$ (following mainly \cite{efz}).
\par
We first note that, according to the estimate \eqref{2.dis1} all of the
trajectories $w_\eb(t)$ and $\widehat w(t)$ are uniformly bounded in
$H^1$. The usual $L^2$-parabolic regularity theorem applied to
equations \eqref{3.eq1} and \eqref{3.aveq1} gives
\begin{equation}\label{3.est1}
\|\Dt w_\eb\|_{L^2((T,T+1),H)}+\|w_\eb\|_{L^2((T,T+1),H^2)}\le C,
\end{equation}
where the constant $C$ depends on the $H^1$-norm of the initial data $\hat w(0)$, but is independent of $\eb$, as $\eb\to0$, and $T\ge0$, and the same
estimate holds for the limit function $\widehat w$ as well.
\par
Let us define  a function $\theta=\theta_\eb(t)$ as the solution
of the following equation:
\begin{equation}\label{3.eqtheta}
\Dt\theta-(1+i\gamma)\Dx\theta+\theta=(1+i\omega)[F(t/\eb,w_\eb)-N(w_\eb)]=:\widetilde H(t/\eb,w_\eb(t)),
\ \theta(0)=0.
\end{equation}
We claim that
\begin{equation}\label{3.est2}
\|\theta_\eb(t)\|_{H^1}\le C\eb, \ \ \text{ for all } t\ge 0,
\end{equation}
for some constant $C$ that is independent of $\eb$. Indeed,
expanding the  function $\widetilde H(\tau,w_\eb(t,x))$ into the Fourier series
with respect to $\tau$, we have
$$
\widetilde H(\tau,w_\eb(t,x))=\sum_{n\in\Bbb Z}e^{in\tau}F_n(t,x),
$$
where
$$
F_n(t,x):=\frac1{2\pi}\int_{-\pi}^\pi \widetilde H(\tau,w_{\eb}(t,x))e^{-in\tau}\,d\tau.
$$

Moreover, by the definition of the operator $N$,
$$
F_0=0
$$
and, by the Parseval identity
$$
\sum_{n\in\Bbb Z}|F_n(t,x)|^2=2\pi\int_{-\pi}^{\pi}|\widetilde H(\tau,w_{\eb}(t,x))|^2\,d\tau,\ \ \sum_{n\in\Bbb Z}|\partial_{x} F_n(t,x)|^2=2\pi\int_{-\pi}^{\pi}|\partial_{x}\widetilde H(\tau,w_{\eb}(t,x))|^2\,d\tau.
$$
Integrating these equalities in $x$, we get
\begin{equation}\label{add-pars}
\sum_{n\in\Bbb Z}\|F_n(t)\|_{H^1}^2= 2\pi\int_{-\pi}^\pi\|\widetilde H(\tau,w_{\eb}(t))\|_{H^1}^2\,d\tau.
\end{equation}
Since $w_\eb$ are uniformly bounded in $H^1$, we have
\begin{equation}\label{3.sum1}
\sum_{n\in\Bbb Z}\|F_n(t)\|_{H^1}^2\le C<\infty,
\end{equation}
where the constant $C$ is independent of $\eb$ and $t$
(here we have implicitly used \eqref{1.alg1}). Furthermore, using
the fact that $\Dt w_\eb(t)$ are uniformly bounded in
$L^2([T,T+1], H)$, and $w_\eb(t)$ are uniformly bounded in $L^2([T,T+1], H^2)$, with respect to $\varepsilon$, and the explicit formula \eqref{1.non1} for
the nonlinearity $F$, arguing analogously, we can prove that
\begin{equation}\label{3.sum2}
\sum_{n\in\Bbb Z}\|\Dt F_n(\cdot)\|_{L^2([T,T+1], H)}^2\le C,
\end{equation}
where the constant $C$ is independent of $\eb$ and $T$.
\par
We fix approximate solution $\widehat\theta$ of \eqref{3.eqtheta} in
the form
\begin{equation}\label{3.fast}
\widehat\theta(t):=\sum_{n\ne0}e^{int/\eb}(-(1+i\gamma)\Dx+1+ in/\eb)^{-1}F_n(t)
\end{equation}
(we solve equation \eqref{3.eqtheta} with respect to the
``fast" variable considering the ``slow" variable $t$ as a
parameter). Then, since the operator $-(1+i\gamma)\Dx$ generates
an analytic semigroup in $H^s$, we have
\begin{equation}\label{3.ans}
\|(-(1+i\gamma)\Dx+1+in/\eb)^{-1}\|_{H^s\to H^s}\le
C_s\frac{\eb}{n},
\end{equation}
see, e.g., \cite{he}.
Estimate \eqref{3.ans} together with \eqref{3.sum1} gives
\begin{equation}\label{3.theta1}
\|\widehat\theta(t)\|_{H^1}\le C\eb\sum_{n\ne0}\frac{\|F_n(t)\|_{H^1}}n\le
C\eb\(\sum_{n\ne0}\frac1{n^2}\)^{1/2}\(\sum_{n\ne0}\|F_n(t)\|_{H^1}^2\)^{1/2}\le
C'\eb.
\end{equation}
We now set $\tilde\theta:=\theta-\widehat\theta$. Then, this function
solves
\begin{equation}\label{3.eqtheta1}
\Dt\tilde\theta-(1+i\gamma)\Dx\tilde\theta+\tilde\theta=h(t), \ \
\tilde\theta(0)=-\widehat\theta(0),
\end{equation}
where
$$
h(t):=-\sum_{n\ne0}e^{int/\eb}(-(1+i\gamma)\Dx+1+ in/\eb)^{-1}\Dt
F_n(t).
$$
Using now estimate \eqref{3.sum2} together with \eqref{3.ans}, we
obtain
\begin{multline}\label{3.theta2}
\|h\|_{L^2([T,T+1],H)}\le
C\eb\sum_{n\ne0}\frac{\|\Dt F_n\|_{L^2([T,T+1],H)}}n\le\\\le C\eb\(\sum_{n\ne0}\frac1{n^2}\)^{1/2}
\(\sum_{n\ne0}\|\Dt F_n(t)\|_{L^2([T,T+1],H)}^2\)^{1/2}\le
C'\eb,
\end{multline}
where the constant $C'$ is independent of $\eb$ and $T$. Estimates
\eqref{3.theta1} and \eqref{3.theta2} show that the solution
$\tilde\theta$ of equation \eqref{3.eqtheta1} satisfies
$$
\|\tilde\theta(t)\|_{H^1}\le C\eb,
$$
where the constant $C$ depends on the initial condition $w_0$, but is independent of
$\eb$ and $t$. Thus, since $\theta=\widehat\theta+\tilde\theta$,
estimate \eqref{3.est2} is indeed satisfied.
\par
Now, we are ready to finish the proof of the theorem. Let $\tilde
w:=w_\eb-\widehat w+\theta$ where $\theta$ solves \eqref{3.eqtheta}. Then,
this function satisfies the following equation
\begin{equation}\label{3.eqtilde}
\Dt\tilde w-(1+i\gamma)\Dx \tilde w=\beta\tilde
w-(\beta+1)\theta+(1+i\omega)[N(\widehat w)-N(w_\eb)], \ \ \tilde
w(0)=0.
\end{equation}
Multiplying this equation by $\Dx \tilde w$ and integrating over
$x$, we deduce that
\begin{equation}\label{3.est4}
\frac{d}{dt}\|\tilde w\|^2_{H^1}\le C(\|\tilde
w\|^2_{H^1}+\|\theta\|^2_{H^1}+\|N(w_\eb)-N(\widehat w)\|^2_{L^2}).
\end{equation}
Using that $w_\eb$ and $\widehat w$ are uniformly bounded in $H^1$,
one can easily see that
$$
\|N(w_\eb)-N(\widehat w)\|_{L^2}\le C\|\tilde w\|_{H^1}+\|\theta\|_{H^1}
$$
and, consequently,
$$
\frac d{dt}\|\tilde w\|^2_{H^1}\le K\|\tilde
w\|^2_{H^1}+C\|\theta\|^2_{H^1}
$$
for some $K$ and $C$ independent of $t$ and $\eb$. Applying the
Gronwall's inequality for that relation and using \eqref{3.est2}, we
deduce estimate \eqref{3.dist} and finish the proof of the
theorem.
\end{proof}
As a standard corollary of this theorem (see, for instance, \cite{cv} or \cite{ha}), we obtain the following
result on the convergence of global attractors $\Cal A_{GL1}(L)$
and $\Cal A_{GL2}(L)$ as $L\to\infty$.

\begin{corollary}\label{Cor3.attrconv} Let the assumptions of Theorem \eqref{Th3.average}
hold, and let $\widehat{\Cal A}_{GL1}$ and $\widehat{\Cal A}_{GL2}$ be the
global attractors of equations \eqref{3.aveq1} and \eqref{3.aveq2},
respectively. Then the family of global attractors $\Cal
A_{GL1}(L)$ (resp. $\Cal A_{GL2}$) of equations \eqref{0.eq1}
(resp. \eqref{0.eq2}) converge, as $L\to\infty$, to the global
attractors $\widehat{\Cal A}_{GL1}$ (resp. $\widehat{\Cal A}_{GL2}$)
associated with limit equations \eqref{3.aveq1} (resp. \eqref{3.aveq2})
in the sense of the upper semi-continuity in $H^s$, $s\ge0$:
\begin{equation}\label{3.attrconv}
\lim_{L\to\infty}\dist_{H^s}(\Cal A_{GL1}(L),\widehat{\Cal
A}_{GL1})=\lim_{L\to\infty}\dist_{H^s}(\Cal A_{GL2}(L),\widehat{\Cal
A}_{GL2})=0,
\end{equation}
recall that $\dist_V(X,Y)$ denotes the  non-symmetric Hausdorff distance between sets $X$ and $Y$ in the $V$ norm.
\end{corollary}
\begin{proof} Let $u_0\in\Cal A_{GLi}$ for $i=1$ or $i=2$. We need to estimate the distance between $u_0$ and the global attractor $\hat{\Cal A}_{GLi}$. According to \eqref{add-str}, there exists a complete bounded trajectory $u_\eb(t)$, $t\in\R$, such that $u_\eb(t)\in\Cal A_{GLi}$, for all
$t\in\R$. Moreover, due to \eqref{2.unattr}, this trajectory is uniformly bounded by a constant independent of $\eb=1/L$ in any space $H^s$.
Let also $w_\eb(t):=\Cal H_{L}(-t)u_\eb(t)$ (resp. $w_\eb(t):=\Cal F_L(-t)u_\eb(t)$). Then, clearly
$$
u_\eb(0)=w_\eb(0)=u_0
$$
and $w_\eb(t)$, $t\in\R$, solves equation \eqref{3.eq1} (resp. \eqref{3.eq2}). Take now an arbitrary $T>0$ and consider the solution $\hat w(t)$, for $t\ge -T$, of the averaged equation \eqref{3.aveq1} (resp. \eqref{3.aveq2}) with the initial data
$$
\hat w(-T)=w_\eb(-T).
$$
Then, due to estimate \eqref{3.dist}, we have
\begin{equation}\label{add-adist}
\|u_0-\hat w(0)\|_{H^s}\le \|w_\eb(0)-\hat w(0)\|_{H^s}\le C_s\eb e^{KT},
\end{equation}
where the constants $C$ and $K$ depend only on $s$. On the other hand, since the limit averaged equation possesses a global attractor in $H^s$, for any $\delta>0$ we may find $T=T(\delta,s)$ such that
$$
d_{H^s}(\hat w(0),\hat {\Cal A}_{GLi})\le \delta/2
$$
for $i=1$ or $i=2$, respectively. Finally, fix $\eb\in(0,\eb_0]$ where $\eb_0>0$ is small enough that $C_s\eb_0 e^{KT}\le\delta/2$, and using the triangle inequality, we see that
$$
d_{H^s}(u_0,\hat{\Cal A}_{GLi})\le \delta/2+\delta/2=\delta,
$$
and since $\delta>0$ and $u_0$ are arbitrary, this inequality proves the desired convergence \eqref{3.attrconv}, which finishes the proof of the corollary.
\end{proof}

We now turn to the case of Kuramoto-Sivashinsky equation
\eqref{0.eq3}. In contrast to the case of cubic nonlinearities,
the quadratic Kuramoto-Sivashinsky nonlinearity disappears after
the averaging and we end up with the non-dissipative averaged
equation. This explains why the norm of the global attractor grows with $L$, as
$L\to\infty$. To be more precise, the following result holds.

\begin{theorem}\label{Th3.KSattr} Assume that $a\ne k^2$, $k\in\Bbb N$. Then,
for every $R>0$ there exists
$L_0=L_0(R)$ such that $u(t)\equiv0$ is the only complete trajectory
on the global attractor $\Cal A_{KS}(L)$,  for $L\ge L_0$ which belongs to the
$R$-ball in $H$ for all $t\in\R$, i.e, the inequality
\begin{equation}\label{3.estks1}
\|u(t)\|_{H}\le R, \ \ \ t\in\R,
\end{equation}
where $u\in\Cal A_{KS}(L)$,
implies that $u\equiv0$.
\end{theorem}
\begin{proof} Let $u(t)$ be a complete trajectory of the KS
equation \eqref{0.eq3} satisfying \eqref{3.estks1}. Then,
according to the smoothing property, see \eqref{2.estsmooth}, we
infer that
\begin{equation}\label{3.estks2}
\|u(t)\|_{H^s}\le R_s, \ \ \ t\in\R,
\end{equation}
where the constants $C_s$ depend on $s$, but are independent of $L$.
\par
Let us now introduce a new dependent variable $w(t):=\Cal
H_L(-t)u(t)$. Then, equation \eqref{0.eq3} reads
\begin{equation}\label{3.eq3}
\Dt w=-\px^4 w-a\px^2 w+ \widetilde H(t/\eb,w),
\end{equation}
where the operator $\widetilde H$ is defined by \eqref{1.non3}. Using now the
fact that $u$ is real, we obtain that $w_{-n}=\bar w_n$ and, since
$\<u\>_{sp}=0$, formula \eqref{1.av3re} gives that the $\tau$-average
of $\widetilde H$ equals zero. Thus, the limit averaged equation for \eqref{3.eq3}
reads
\begin{equation}\label{3.aveq3}
\Dt\widehat w=-\px^4\widehat w-a\Nx^2 \widehat w.
\end{equation}
Let us fix $\tau\in\R$ and a solution $\hat w(t)$, $t\ge\tau$ of equation \eqref{3.aveq3} such that
$$
\hat w(\tau)=w(\tau).
$$
Then, using estimate \eqref{3.estks2} and arguing exactly as
in Theorem \ref{Th3.average}, we establish that, for any $\tau\in\R$ and $t\ge0$, the following estimate holds:
\begin{equation}\label{3.estks4}
\|w(t+\tau)-\widehat w(t+\tau)\|_H\le C\eb e^{Kt},\ t\ge0,
\end{equation}
where the positive constants $C$ and $K$ are independent of $t$,
$\tau$ and $\eb:=1/ L$.
\par
However, the situation is principally different from the case of
Ginzburg-Landau equations, since the averaged equation is now
linear and have exponentially growing modes. We claim that
inequalities \eqref{3.estks4} and \eqref{3.estks1} imply the
estimate
\begin{equation}\label{3.kssmall}
\|u(t)\|_H\le C\eb,\ \ t\in\R,
\end{equation}
for some positive $C$  which may depend on $R$, but
is independent of $\eb\to0$. Instead of proving estimate \eqref{3.kssmall} for the function $u$, we will prove its analogue for $w$ which is equivalent to \eqref{3.kssmall} since $\mathcal H_L$ is an isometry in $H$.
\par
Indeed, let $P_N$ be an orthogonal projector onto
the unstable modes of equation \eqref{3.aveq3} (i.e., to the
vectors $e_n$ with $n^2<a$) and $Q_N:=1-P_N$ is the associated
projector to the stable ones (we recall that $\sqrt a\notin\Bbb
N$, so the equilibrium $\widehat w=0$ is hyperbolic). Thus, the
solution $\widehat w$ satisfies the standard exponential dichotomy
estimates
\begin{equation}\label{3.diest}
\begin{cases}
\|P_N\widehat w(t+\tau)\|_H\ge C\|P_N\widehat w(\tau)\|_He^{\alpha t},\ t\ge0,\\
\|Q_N\widehat w(t+\tau)\|_H\le C\|Q_N\widehat w(\tau)\|_He^{-\alpha t},\ t\ge0,
\end{cases}
\end{equation}
for some positive $C$ and $\alpha$. Let us first consider
$P_N$-component of $w$. Then, using the solution $\widehat w$  and estimate \eqref{3.diest}, we get
$$
\|P_N w(t+\tau)\|_{H}\ge\|P_N\hat w(t+\tau)\|_H-\|w(t+\tau)-\hat w(t+\tau)\|_H\ge C\|P_N w(\tau)\|_He^{\alpha t}-C\eb e^{Kt}
$$
which gives
$$
C\|P_N w(\tau)\|_He^{\alpha t}\le C\eb e^{Kt}+\|P_N w(t+\tau)\|_H.
$$
Taking now the supremum over $\tau\in\R$ from the both sides of
that inequality, we arrive at
\begin{equation}\label{3.pnsmall}
\sup_{\tau\in\R}\|P_N w(\tau)\|_H\le C\eb\frac{e^{Kt}}{Ce^{\alpha t}-1}.
\end{equation}
Fixing here $t$ in an optimal way, we obtain the analogue of
\eqref{3.kssmall}
for the $P_N$-component of $w$:
\begin{equation}\label{3.smallpn}
\sup_{\tau\in\R}\|P_N w(\tau)\|_H\le C^*\eb,
\end{equation}
where $C^*$ is independent of $\eb$.
\par
Let us consider now the $Q_N$-component of $w$. Then, the
analogous arguments lead to the inequality
$$
\|Q_N w(t+\tau)\|_H\le C_1\|Q_N w(\tau)\|_He^{-\alpha t}+C_2\eb
e^{Kt}.
$$
Taking again the supremum over $\tau\in\R$, we arrive at
$$
\sup_{\tau\in\R}\|Q_N w(\tau)\|_H\le
C_2\eb\frac{e^{Kt}}{1-C_1e^{-\alpha t}}, \text{ for all } t>\frac{\ln(C_1+1)}{\alpha}.
$$
Minimizing the right-hand side with respect to $t$ over $[ln(C_1+ 1)/\alpha, \infty)$ we obtain
\begin{equation}\label{3.smallqn}
\sup_{\tau\in\R}\|Q_N w(\tau)\|_H\le C^*\eb.
\end{equation}
Thus, estimate \eqref{3.kssmall} is verified.
\par
We see that any trajectory $w(t)$ of equation
\eqref{3.eq3} which is contained in the ball of radius $R$ in $H$ should belong to the $C\eb$-neighbourhood of zero
equilibrium $w\equiv0$  (where the constant $C$ depends on $R$,
but is independent of $\eb=1/L$). It remains to note that, since
the
equilibrium $w\equiv0$ is hyperbolic and $\tilde H(\tau,0)=0$, $D_w\tilde H(\tau,w)|_{w=0}=0$, the usual implicit function theorem shows that
the only solution $w(t)$ of \eqref{3.eq3} which belongs to some small
$r$-ball of $H$ for any $t$ ($r$ is independent of $\eb$!) is
$w\equiv0$. Thus, $u\equiv0$ if $\eb$ is small enough and the
theorem is proved.
\end{proof}

The proof of Theorem \ref{Th3.KSattr} indicates particularly that the diameter of  the global attractor $\mathcal A_{KS}(L)$ (say, in the $H$-norm) indeed expands as $L\to\infty$:
\begin{equation}\label{new.4.attr}
\lim_{L\to\infty}\|\mathcal A_{KS}(L)\|_H=\infty,
\end{equation}
but gives no information about the rate of expansion. The next proposition removes this drawback and shows that the upper bound given by estimate \eqref{2.dis3} is optimal. The proof is based on the well-known fact on the existence of rotating waves for the perturbed KdV equation with periodic boundary conditions, see \cite{Oga} (see also \cite{erc} for the numerical study of the stability of these waves as well as the related attractors).

\begin{proposition}\label{Prop4.KSlow} Let $a>1$ and $L$ be large enough. Then, the global attractor $\mathcal A_{KS}(L)$ satisfies the following estimate:
\begin{equation}
C_1 L\le \|\mathcal A_{KS}(L)\|_H\le C_2 L,
\end{equation}
where the positive constants $C_1, C_2$ are independent of $L$.
\end{proposition}
\begin{proof} Indeed, the upper bound is an immediate corollary of estimate \eqref{2.dis3}, so we only need to establish the lower one.
To this end, we do change of variables $u(t)=Lv(t)$ where $v$ is a new dependent variable. Then, equation \eqref{0.eq3} reads
$$
\eb \Dt v=\eb(-\px^4 v-a\px^2 v)+v\px v+\px^3 v,
$$
and, finally, introducing the fast variable $\tau=Lt$, we end up with
\begin{equation}\label{new.KdV}
\partial_\tau v=\eb(-\px^4 v-a\px^2 v)+v\px v+\px^3 v,
\end{equation}
which is a  well studied small damped-driven perturbation
of the Korteweg de Vries equation. In particular, as shown in \cite{Oga}, if $a>1$ and $\eb$ is small enough, there is a rotating wave solution
$$
v(\tau,x)=V(x-c \tau)
$$
of this equation, where $V=V_\eb(\xi)$ is a $2\pi$-periodic function with zero mean and $c=c_\eb$ is a wave speed, both of which have finite non-zero limits as $\eb\to0$. Since this rotating wave obviously belongs to the global attractor, it gives the desired lower bound and finishes the proof of the proposition.
\end{proof}

\section{Reduction to the gradient case}\label{s4}

The main aim of this section is to give more comprehensive study of the dynamics of the complex Ginzburg-Landau equation \eqref{0.eq2} for large dispersion parameter $L$.
As we will see, the large dispersion suppresses the non-trivial dynamics and makes the system gradient-like up to some isometric transformation.
\par
Indeed, the limit averaged equation for \eqref{0.eq2} has the following form, see \eqref{3.aveq2}:
\begin{equation}\label{4.aver2.gen_f}
\frac d{dt}w_n=-n^2w_n+\beta w_n - (1 + i\omega)(2 w_n\|w\|^2_H - w_n|w_n|^2),\ \ \ n\in\Bbb Z,
\end{equation}
where $w(t)=\sum_{n\in\Bbb Z}w_n(t)e_n$ and $\beta = \alpha + i \gamma$.
\par
Introducing the new variables $v_n$ such that $w_n = e^{i A_n(t)}v_n,  n\in \Bbb Z$, and the phases $A_n(t)$ will be determined later, we get
\begin{equation}\label{4.aver2.tr}
i v_n \frac{d}{dt}A_n(t)+ \frac{d}{dt}v_n = - n^2 v_n + (\alpha + i \gamma)v_n - (1 + i \omega)v_n(2\|v\|^2_H + |v_n|^2), \ \ \ n\in \Bbb Z.
\end{equation}
Thus,
 if we fix the phases $A_n$ as follows:
 \begin{equation}\label{5.new.ph}
 A_n(t): = \int_0^t(\gamma - 2\omega \|v\|^2_H +\omega|v_n|^2)dt=\int_0^t(\gamma - 2\omega \|w\|^2_H +\omega|w_n|^2)dt,
 \end{equation}
we arrive to the equation with real coefficients

\begin{equation}\label{4.aveq2}
\frac d{dt}v_n=-n^2v_n+\alpha v_n-2v_n\|v\|_H^2+v_n|v_n|^2,\ \ \ n\in\Bbb Z.
\end{equation}
Moreover, any solution $w(t)$ of \eqref{4.aver2.gen_f} determines in a unique way the corresponding solution $v(t)$ of \eqref{4.aveq2} and vice versa. Therefore, it is sufficient to study equations \eqref{4.aveq2} only.
\par
Furthermore, equations \eqref{4.aveq2} possess a global Lyapunov function
\begin{equation}\label{4.lyap}
\Cal L(v):=\|v\|^2_{H^1}+\|v\|^4_H-\sum_{n\in\Bbb Z}\frac12|v_n|^4 - (\alpha + 1)\|v\|^2_H.
 \end{equation}
Indeed, as direct calculations show
\begin{equation}\label{4.dt}
\frac d{dt}\Cal L(v(t))=-2\sum_{n\in\Bbb Z}\frac{d}{dt}v_n \frac{d}{dt}\overline{v}_n=-2\sum_{n\in\Bbb Z}|\frac d{dt}v_n|^2.
\end{equation}
We are going to apply what is called regular attractors theory (see, e.g. \cite{bv}) in order to describe
the global attractor $\widehat{\Cal A}_{GL2}$ of the limit equation \eqref{3.aveq2} and the global attractor $\Cal A_{GL2}(L)$  of the perturbed system \eqref{3.eq2}, for $L$ is large enough.
However, our situation is slightly more complicated in comparison with the standard theory since equation \eqref{4.aveq2}
possesses a huge symmetry group and, for this reason,  all of the equilibria are degenerate. Indeed, it follows from the structure
of \eqref{4.aveq2} that the group $\R^\infty$ acting on the phase space by
\begin{equation}\label{4.symgr}
[\Cal S(\phi)v]_n=e^{i\phi_n}v_n,\ \ \ \phi_n\in\R,\ \ \phi=\{\phi_n\}_{n\in\Bbb Z}\in\Bbb R^\infty
\end{equation}
is a symmetry group of \eqref{4.aveq2}. For this reason, if $v$ is an equilibrium of equation \eqref{4.aveq2}, then we automatically have the
whole torus of equilibria generated by this symmetry group:
\begin{equation}\label{4.tori}
\Bbb T_{v}:=\{\Cal S(\phi)v,\ \ \phi\in\R^\infty\}
\end{equation}
(due to Proposition \ref{Prop3.inv}, $v_n\equiv 0$ for $|n|\ge N=N(\beta)$, for any equilibrium and, consequently,
all these tori are, in a fact, finite-dimensional). Thus, the assumption on the hyperbolicity of equilibria should
be naturally replaced by the assumption that all of the equilibria tori \eqref{4.tori} are normally hyperbolic.
\par
In order to verify this normal hyperbolicity and the consequent structure of the global attractor $\widehat{\Cal A}_{GL2}$ of the limit equation \eqref{3.aveq2},
we use a simple observation that a real hyperplane
\begin{equation}\label{4.real}
H^{real}:=\{v\in H,\ v_n\in\Bbb R,\ n\in\Bbb Z\}
\end{equation}
is invariant with respect to the limit equation \eqref{4.aveq2}. Moreover, every initial data $v\in H$ can be
reduced to this hyperplane by the appropriate action of the symmetry group $\Cal S(\phi)$. Thus, it is sufficient
to check all of the hyperbolicity assumptions for the case of {\it real} equations \eqref{4.aveq2} and after that obtain
the required result for the initial complex phase space by the action of the symmetry group. In particular, one has a natural
relation between the real and complex attractors:
\begin{equation}\label{4.rc}
\widehat{\Cal A}_{GL2}=\{\Cal S(\phi)\widehat {\Cal A}^{re}_{GL2},\ \phi\in\R^\infty\}.
\end{equation}
The next lemma gives an explicit description of all possible equilibria for problem \eqref{4.aveq2} and establishes their
hyperbolicity for generic $\beta$.

\begin{lemma} \label{Lem4.eqdescription} The set  $\Cal R$ of all equilibria of equation \eqref{4.aveq2} consists of $v = 0$ and $v\ne0$ such that the non-zero components $\{v_n\}_{n\in\Bbb Z}$
possess the following description: let
$$
N_0:=\sum_{n,v_n\ne0}1,\ \ \ N_2:=\sum_{n,v_n\ne0}n^2.
$$
Then, for $v_n\ne0$,
\begin{equation}\label{4.equII}
\begin{cases}
\|v\|^2_H=\frac {N_0}{2N_0-1}\alpha-\frac{N_2}{2N_0-1}>0,\\
|v_n|^2=n^2+\frac{\alpha-2N_2}{2N_0-1}>0
\end{cases}
\end{equation}
and every sequence of $v_n$ satisfying these conditions gives an equilibrium.
\par
Moreover, a non-zero equilibrium $v$ is not hyperbolic (i.e., the corresponding torus is not normally hyperbolic) if and only if
\begin{equation}\label{4.2hyp}
k^2+\frac{\alpha-2N_2}{2N_0-1}=0 \ \ \text{for some $k$ for which $v_k=0$}.
\end{equation}
Finally, zero equilibrium is hyperbolic if and only  if  $\alpha\ne k^2$ for some $k\in\Bbb Z$. In particular, all of the equilibria are hyperbolic if $\alpha\ne\Bbb Z$.
\end{lemma}
\begin{proof} Indeed, let $v$ be a non-zero equilibrium. Then, equations \eqref{4.aveq2} for  $v_n\ne0$ are equivalent to
\begin{equation}\label{4.nII}
-n^2+\alpha-2\|v\|_H^2+|v_n|^2=0.
\end{equation}
Solving these equations (by using that $\|v\|_H^2=\sum_{n,v_n\ne0}|v_n|^2$), we obtain \eqref{4.equII}.
\par
Let us now study the hyperbolicity. As we have already mentioned before, to this end, it is sufficient to consider the case of {\it real}
equilibrium $v\in\Cal R\cap H^{real}$ and real perturbation $\theta\in H^{real}$. For that class of perturbations, the equation of variations reads
\begin{equation}\label{4.eqlin}
\frac {d}{dt} \theta_n =-n^2\theta_n+\alpha\theta_n-2\theta_n\|v\|_H^2-4v_n(v,\theta)+3\theta_nv_n^2,\ \
n\in \Bbb Z.
\end{equation}
Let us try to find a non-zero eigenvector $\theta\in H^{real}$ for the right-hand side of these equations which would correspond to the zero eigenvalue.

For $n$th with $v_n\ne0$, equations \eqref{4.nII} allow to transform the equilibria equation
for \eqref{4.eqlin} as follows
$$
-4v_n(v,\theta)+2\theta_nv_n^2=0
$$
or, since $v_n\ne0$, this gives
\begin{equation}\label{4.strange}
\theta_nv_n=2(v,\theta).
\end{equation}
Taking a sum of that equations, we see that $(v,\theta)=2N_0(v,\theta)$ and consequently $(v,\theta)=0$. Equation \eqref{4.strange} now gives
that $\theta_n=0$ for all $n$ such that $v_n\ne0$. Let us now consider $k\in\Bbb Z$ such that $v_k=0$. Then, equation for $\theta_k$ reads
$$
\theta_k(-k^2+\alpha-2\|v\|_H^2)=\theta_k(-k^2-\frac{\alpha-2N_2}{2N_0-1})=0
$$
which implies $\theta_k=0$ if \eqref{4.2hyp} is not satisfied. This implies $\theta\equiv0$ and non-zero equilibrium $v$ is hyperbolic. Moreover, we see that
\eqref{4.2hyp} cannot be true if $\alpha$ is not integer, so, for the non-integer $\alpha$ any non-zero equilibrium $v$ is automatically hyperbolic.
\par
Finally, the assertion about zero equilibrium is evident and Lemma \ref{Lem4.eqdescription} is proved.
\end{proof}
The next lemma gives the stability of the equilibria found in the previous lemma.

\begin{lemma}\label{Lem5.new.stable} Let $\alpha>0$. Then all equilibria $v\in\mathcal R$ which have two or more non-zero components ($v_m\ne0$ and $v_k\ne0$ for some $m\ne k$) are unstable. The only
stable equilibria are the one component ones: $v_n=0$ for $n\ne k$ $v_k\ne0$ and $k^2<\frac\alpha2$.
\end{lemma}
\begin{proof} Let us prove that any equilibrium with two or more non-zero components is unstable. Indeed, let $v\in\mathcal R$ be such that $v_k\ne0$ and $v_m\ne0$ for some $k\ne m$. Then, equation \eqref{4.nII} holds for $n=k$ and $n=m$. Then, \eqref{4.eqlin} for these components reads
\begin{equation}
\frac{d}{dt}\theta_k=2\theta_k|v_k|^2-4v_k(v,\theta),\ \ \frac{d}{dt}\theta_m=2\theta_m|v_m|^2-4v_m(v,\theta)).
\end{equation}
Thus, if we take an inner product of the right-hand side of \eqref{4.eqlin} with the non-zero vector $\theta$ such
that $\theta_n=0$ if $n\ne k,m$ satisfying $(v,\theta)=0$ (exactly for this reason we need at least two non-zero components of $v$), the result will be strictly positive: $\theta_k^2v_k^2+\theta_m^2v_m^2>0$. By the min-max theorem, this means that the corresponding equilibrium is unstable.
\par
Let us now study the equilibria with only one non-zero component $v_k\ne0$ for some $k\ne0$. Then, \eqref{4.nII} gives
\begin{equation}
|v_k|^2=\alpha-k^2>0.
\end{equation}
Using this equation in order to simplify \eqref{4.eqlin} with $n\ne k$ and using that $v_n=0$, we have
$$
\frac d{dt}\theta_n=\theta_n(\alpha-2|v_k|^2-n^2)=\theta_n(2k^2-\alpha-n^2).
$$
Finally, the equation for the component with $n=k$ reads
$$
\frac d{dt}\theta_k=-2\theta_k v_k^2.
$$
Thus, the linearization \eqref{4.eqlin} at such equilibria is diagonal. Moreover, obviously, all entries on the diagonal will be negative if and only if $2k^2-\alpha<0$, and the lemma is proved.
\end{proof}
As a standard corollary of this lemma and the existence of a global Lyapunov function, we obtain the following
result.
\begin{theorem}\label{Th4.rattr} Let the parameter $\alpha$ be such that all of the equilibria
$v\in\Cal R$ of the equation \eqref{4.aveq2} are hyperbolic (see Lemma \ref{Lem4.eqdescription}).
Then,
\par\noindent
1) The unstable set of any equilibrium $v_0\in\Cal R$
\begin{equation}\label{4.man}
\Cal M_{v_0}^+:=\{v\in H^1, \ \exists v(t),\ t\in\R, \ \text{$v$ solves
\eqref{4.aveq2}},  \ v(0)=v, \ \ \lim_{t\to-\infty}\dist(v(t),\Bbb
T_{v_0})=0\}
\end{equation}
is a finite-dimensional submanifold of the phase space (say, $H^1$).
\par\noindent
2) The global attractor $\widehat{\Cal A}_{GL2}$ for equations \eqref{4.aveq2} is a finite union of
the finite-dimensional unstable manifolds:
\begin{equation}\label{4.regattr0}
\widehat{\Cal A}_{GL2}=\cup_{v_0\in\Cal R}\Cal M^+_{v_0}.
\end{equation}
3) Any trajectory $v(t)$, $t\in\R$, belonging to the global attractor is a
heteroclinic orbit between two equilibria $v_0^+,v_0^-\in\Cal R$
belonging to different tori:
\begin{equation}\label{4.hetero}
\lim_{t\to+\infty}\|v(t)-v_0^+\|_{H^1}=\lim_{t\to-\infty}\|v(t)-v_0^-\|_{H^1}=0,\
\ \Bbb T_{v_0^+}\ne\Bbb T_{v_0^-}
\end{equation}
4) The global attractor $\widehat{\Cal A}_{GL2}$ is exponential, i.e., there
exists a positive constant $\alpha>0$ and a monotone function $Q$
such that
\begin{equation}\label{4.expattr}
\dist_{H^1}(S(t)B,\widehat{\Cal A}_{GL2})\le Q(\|B\|_{H^1})e^{-\alpha t}
\end{equation}
for any bounded subset $B$ of the phase space $H^1$.
\end{theorem}
Indeed, the standard regular attractors theory (see, e.g., \cites{bv,vcz}) can be applied for the {\it real-valued} version of
equations \eqref{4.aveq2} where all of the equilibria are
hyperbolic in a usual sense and the general case can be treated
after that by expression \eqref{4.rc}.
\par
Let us return now to the initial equation \eqref{4.aver2.gen_f}. To this end, we just need to put $w_n(t)=e^{iA_n(t)}v_n(t)$, where the phases $A_n(t)$ are defined by \eqref{5.new.ph}. In particular, any equilibrium $v\in\mathcal R$ generates a quasi-periodic solution of \eqref{4.aver2.gen_f}. Indeed, in that case
$$
A_n(t)=(\gamma+\omega(n^2-\alpha))t+C_n,\ \ v_n\ne0,
$$
and we see the quasi-periodic motion with no more than two independent frequencies (generated by $\gamma-\omega\alpha$ and $\omega$). Thus, invariant tori of equilibria for equation \eqref{4.aveq2} correspond to the same tori, but filled by quasi-periodic motions on the level of equation \eqref{4.aver2.gen_f}. Analogously, any heteroclinic orbit connecting the equilibria of equation \eqref{4.aveq2} corresponds to the heteroclinic orbits between the aforementioned quasi-periodic motions on the invariant tori. Thus,  Theorem \ref{Th4.rattr} extends to the initial equation \eqref{4.aver2.gen_f} just by replacing the tori of equilibria by the tori filled by the aforementioned quasi-periodic motions.

We conclude the section by treating the non-averaged equation \eqref{3.eq2} as a
small (of order $\eb=1/L$ according to Theorem \ref{Th3.average})
perturbation of the limit equation \eqref{4.aveq2}. For simplicity, we restrict ourselves to the gradient case $\beta=\alpha\in\R$, $\omega=0$, see Remark \ref{Rem4.verylast} below, concerning the general case.
\par
 We first note
that the non-averaged equations do not possess the symmetry group
\eqref{4.symgr} and, consequently the {\it equilibria} tori $\Bbb
T_{w_0}$  disappear (in general) for the
perturbed equations \eqref{3.eq2}.
\par
However, since the (finite-dimensional) invariant tori $\Bbb T_w$, $w\in\Cal R$,
 are normally hyperbolic, they preserve under the small
 perturbations. To be more precise, in the $\eb$-neighbourhood of
 every the non-perturbed torus $\Bbb T_{w_0}\sim \Bbb T^N$ there exists a
 smooth ($C^k$-smooth) invariant torus $\Bbb T_{w_0}(\eb)$
 of the perturbed system \eqref{3.eq2} if $\eb>0$ is small enough.
  In contrast to the non-perturbed case, the points $w=w(\phi)\in\Bbb T_{w_0}(\eb)$, $\phi\in\Bbb T^N$ are no
 more equilibria, but evolve slowly in time and this evolution is
 governed by the appropriate system of ODEs
\begin{equation}\label{4.torper}
 \frac d{dt}\phi=\eb f_\eb(t,\phi),\ \ \ \phi\in\Bbb T^N
\end{equation}
where the vector field $f_\eb$ is of order one as $\eb\to0$:
$$
\|f_\eb(t,\cdot)\|_{C^k(\Bbb T^N)}\le C.
$$
We also mention that, although the non-perturbed system
\eqref{3.eq2} depends explicitly on time, the invariant tori $\Bbb
T_{w_0}(\eb)$ are independent of time since this equation is
invariant under the $\Cal F_L(s)$-transformations.
\par
Furthermore, according to the general theory, the normally hyperbolic invariant manifolds
$\Bbb T_{w_0}(\eb)$ possess the unstable manifolds $\Cal
M_{w_0}^+(\eb)$ which are $\eb$-close to the unstable manifolds
$\Cal M^+_{w_0}$ of the limit system (moreover, by the above
mentioned reasons, they are also independent of $t$).
\par
Finally, applying the standard perturbation theory of regular
attractors we end up with the following result.
\begin{theorem}\label{Th4.regattrper} Let the assumptions of
Theorem \ref{Th4.rattr} hold and let in addition $\beta = \alpha \in \Bbb{R}$ and $\omega = 0$. Then, there exists $\eb_0>0$ such
that, for any $\eb\in[0,\eb_0]$,
\par\noindent
1) Every equilibria torus $\Bbb T_{w_0}$, $w_0\in\Cal R$,
generates (in an $\eb$-neighbourhood) a normally hyperbolic invariant torus $\Bbb
T_{w_0}(\eb)$ of the perturbed system \eqref{3.eq2} and the dynamics on it is governed by
 the slow equations \eqref{4.torper}.
\par
2) The global attractor $\Cal A_{GL2}(L)$ is a finite union of the
finite-dimensional unstable manifolds $\Cal M_{w_0}^+(\eb)$ to
that tori:
\begin{equation}\label{4.peratrstr}
\Cal A_{GL2}(L)=\cup_{w_0\in\Cal R}\Cal M_{w_0}^+(\eb).
\end{equation}
3) Every trajectory $w(t)$ of the perturbed system \eqref{3.eq2}
is a heteroclinic orbit between two trajectories $w^-(t)\in\Bbb T_{w^-_0}(\eb)$
and $w^+(t)\in\Bbb T_{w^+_0}(\eb)$ belonging to different
invariant tori:
\begin{equation}
\lim_{t\to+\infty}\|w(t)-w^+(t)\|_{H^1}=\lim_{t\to-\infty}\|w(t)-w^-(t)\|_{H^1}=0,\
\ w^\pm(t)\in\Bbb T_{w_0}^\pm(\eb), \ \Bbb T_{w_0^-}\ne\Bbb
T_{w_0^+}.
\end{equation}
4) The global attractor $\Cal A_{GL2}(L)$ is exponential, i.e.,
there exist a positive constant $\alpha$ and a monotone function $Q$, such
that,
for every bounded set $B$ of $H^1$,
\begin{equation}\label{4.perexp}
\dist_{H^1}(S_\eb(t)B,\Cal A_{GL2}(L))\le Q(\|B\|_{H^1})e^{-\alpha t},
\end{equation}
where $S_\eb(t)$ is a solving operator for equation \eqref{3.eq2}.
\par\noindent
5) The global attractors $\Cal A_{GL2}(L)$ of the equation \eqref{3.eq2} tend to the global attractor
$\widehat{\Cal A}_{GL2}$ of the limit equation \eqref{3.aveq2} in the sense of symmetric Hausdorff distance.
Moreover, the following estimate holds
\begin{equation}\label{4.attrsym}
\dist^{sym}_{H^1}(\Cal A_{GL2}(L),\widehat{\Cal A}_{GL2})\le
C\(\frac1L\)^\kappa,
\end{equation}
where the positive constants $C$ and $\kappa$ are independent of
$L$.
\end{theorem}
Indeed, although, in contrast to the general theory, we have now
invariant normally hyperbolic tori instead of hyperbolic
equilibria, the proof of the result repeats word by word the
standard arguments and for this reason it is omitted (see, e.g.,
\cites{bv,efz,vcz} for the details).
\begin{remark}\label{Rem4.last} The last result shows that, for
large dispersion parameter $L$ and $\omega=0$, the dynamics generated by the
Ginzburg-Landau equation \eqref{0.eq2} generates three
different time scales:
\par\noindent
1) rapid oscillations of the phases of the Fourier coefficients (with the frequency proportional to $L$) generated by the group action $\Cal F_L(t)$;
\par\noindent
2) order one heteroclinic motion of their amplitudes close to the limit dynamics of
\eqref{4.aveq2};
\par\noindent
3) slow drift of their phases (of order $\eb=1/L$) governed by equations
\eqref{4.torper}.
We see that the first two types of dynamics are regular. However,
the last small drift on the invariant tori, in principle, may be
chaotic.
\end{remark}
\begin{remark}\label{Rem4.verylast} The analogue of Theorem \ref{Th4.regattrper} remains true in the general case $\beta\in\mathbb C$ and $\omega\ne0$. The only difference is that the invariant tori will be filled not by slow motions, but by the motions $\eb$-close to the quasi-periodic ones.
\end{remark}

\begin{remark}\label{Rem4.turaev} It is natural to ask whether or not the analogous results hold for the  equation \eqref{0.eq1}. In particular, whether or not the averaged equations \eqref{3.aveq1} are in a sense gradient or possess the global Lyapunov function which forbid the complicated dynamics. As shown in \cite{OTZ}, the answer on this question is {\it negative} and the dynamics of \eqref{3.aveq1} is {\it chaotic} at least for some values of the parameters $\gamma$, $\beta$ and $\omega$. Indeed, equations \eqref{3.aveq1} obviously possess the 4-dimensional invariant manifold
\begin{equation}
\hat w=y e_0+v(e_1+e_{-1}),\ \ y,v\in\mathbb C
\end{equation}
and the dynamics on this manifold is given by the equations
\begin{equation}
\begin{cases}
\dot y=\beta y-(1+i\omega)[y(|y|^2+4|v|^2)+2\bar y v^2],\\
\dot v=(\beta-1-i\gamma)v-(1+i\omega)[v(2|y|^2+3|v|^2)+\bar vy^2].
\end{cases}
\end{equation}
Substituting $y=\sqrt{r}e^{i\varphi}$, $v=\sqrt{\rho}e^{i\psi}$, $\eta=2(\psi-\varphi)$ and scaling time by the factor of 2, we arrive at the following 3-dimensional system:
\begin{equation}
\begin{cases}\label{4.terrible}
\dot r=r[\beta-r-4\rho-2\rho(\cos\eta-\omega\sin\eta)],\\
\dot \rho=\rho[\beta-1-2r-3\rho-r(\cos\eta+\omega\sin\eta)],\\
\dot \eta=-\gamma+\omega(\rho-r)+r(\sin\eta-\omega\cos\eta)+2\rho(\sin\eta+\omega\cos\eta)],
\end{cases}
\end{equation}
see \cite{OTZ} for the details. As shown there, there are values of parameters $(\gamma_0,\beta_0,\omega_0)$ for which equations \eqref{4.terrible} possess an equilibrium $(r_0,\rho_0,\eta_0)$ with 3 zero eigenvalues. Then, the standard bifurcation analysis (also performed in \cite{OTZ}) shows that this system possesses the Shilnikov  saddle-focus homoclinic loop at certain values of the parameters and, as a result there is an open region in the space of parameters $(\beta,\gamma,\omega)$ for which the corresponding dynamics generated by \eqref{4.terrible} is chaotic.
\par
 Thus, in contrast to \eqref{3.aveq2}, the averaged equations \eqref{3.aveq1} {\it cannot} be transformed to a gradient system. Moreover, the standard perturbation arguments show that the chaotic dynamics of the limit averaged equations \eqref{3.aveq1} persists at the initial equations \eqref{0.eq1} if $L$ is large enough.
\end{remark}
\renewcommand{\thesection}{}
\renewcommand{\thesubsection}{\arabic{section}.\arabic{subsection}}
\makeatletter
\def\@seccntformat#1{\csname #1ignore\expandafter\endcsname\csname the#1\endcsname\quad}
\let\sectionignore\@gobbletwo
\let\latex@numberline\numberline
\def\numberline#1{\if\relax#1\relax\else\latex@numberline{#1}\fi}
\makeatother

\tocless\section{Acknowledgements}

The work of S.Z. was partially supported by the Russian Foundation of Basic Researches (projects 14-01-00346 and 15-01-03587) and by the grant 14-41-00044 of RSF. The work of E.S.T. was supported in part by the ONR grant N00014-15-1-2333 and the NSF grants DMS-1109640 and DMS-1109645.
\bibliography{References}

\begin{bibdiv}
\begin{biblist}

\bib{W2}{article}{
      author={{Al-Jaboori}, M.},
      author={{Wirosoetisno}, D.},
       title={{Navier-Stokes equations on the $\beta$-plane}},
        date={2011},
     journal={{Discrete Contin. Dyn. Syst. Ser. B}},
      volume={16},
      number={3},
       pages={687–\ndash 701},
}

\bib{arnold}{book}{
      author={{Arnold}, V.},
       title={{Mathematical Methods of Classical Mechanics (Graduate Texts in
  Mathematics, Vol. 60)}},
   publisher={New York, NY: Springer},
        date={1997},
}

\bib{akst}{article}{
      author={{Artstein}, Z.},
      author={{Kevrikidis}, I.},
      author={{Slemrod}, M.},
      author={{Titi}, E.S.},
       title={{Slow observables of singularly perturbed differential
  equations}},
        date={2007},
     journal={{Nonlinearity}},
      volume={20},
       pages={2463–\ndash 2481},
}

\bib{alt}{article}{
      author={{Artstein}, Z.},
      author={{Linshiz}, J.},
      author={{Titi}, E.S.},
       title={{ Young measure approach to computing slowly advansing fast
  oscillations}},
        date={2007},
     journal={{Multiscale Model. Simul.}},
      volume={6},
      number={4},
       pages={1085–\ndash 1097},
}

\bib{BIT}{article}{
      author={{Babin}, A.},
      author={{Ilyin}, A.},
      author={{Titi}, E.S.},
       title={{On the regularization mechanism for the periodic Korteweg-de
  Vries equation}},
        date={2011},
        ISSN={0010-3640; 1097-0312/e},
     journal={{Commun. Pure Appl. Math.}},
      volume={64},
      number={5},
       pages={591\ndash 648},
}

\bib{BMN}{article}{
      author={{Babin}, A.},
      author={{Mahalov}, A.},
      author={{Nicolaenko}, B.},
       title={{Global regularity of 3D rotating Navier-Stokes equations for
  resonant domains}},
        date={2000},
        ISSN={0893-9659},
     journal={{Appl. Math. Lett.}},
      volume={13},
      number={4},
       pages={51\ndash 57},
}

\bib{bv}{article}{
      author={{Babin}, A.},
      author={{Vishik}, M.},
       title={{Regular attractors of semigroups and evolution equations}},
        date={1983},
        ISSN={0021-7824},
     journal={{J. Math. Pures Appl. (9)}},
      volume={62},
       pages={441\ndash 491},
}

\bib{bv1}{book}{
      author={{Babin}, A.},
      author={{Vishik}, M.},
       title={{Attractors of Evolution Equations}},
   publisher={Amsterdam etc.: North-Holland},
        date={1992},
        ISBN={0-444-89004-1/hbk},
}

\bib{BLP}{book}{
      author={{Bensoussan}, A.},
      author={{Lions}, J.-L.},
      author={{Papanicolaou}, G.},
       title={{Asymptotic Analysis for Periodic Structures. Studies in
  Mathematics and Its Applications. Vol. 5}},
   publisher={{North-Holland Publ. Amsterdam - New York - Oxford}},
        date={1978},
}

\bib{bo}{book}{
      author={{Bogolyubov}, N.},
       title={{On Some Statistical Methods in Mathematical Physics}},
   publisher={Kiev: Izd. AN USSR},
        date={1945},
}

\bib{bour1}{article}{
      author={{Bourgain}, J.},
       title={{Fourier transform restriction phenomena for certain lattice
  subsets and applications to nonlinear evolution equations. II. The
  KdV-equation}},
        date={1993},
     journal={{Geom. Funct. Anal.}},
      volume={3},
      number={3},
       pages={209\ndash 262},
}

\bib{bour2}{article}{
      author={{Bourgain}, J.},
       title={{Periodic Korteweg de Vries equation with measures as initial
  data}},
        date={1997},
     journal={{Selecta Math. (N.S.)}},
      volume={3},
      number={2},
       pages={115\ndash 159},
}

\bib{bour}{book}{
      author={{Bourgain}, J.},
       title={{Global Solutions of Nonlinear Schr\"odinger Equations. American
  Mathematical Society Colloquium Publications, Vol. 46}},
   publisher={Providence, RI: AMS},
        date={1999},
}

\bib{cv}{book}{
      author={{Chepyzhov}, V.},
      author={{Vishik}, M.},
       title={{Attractors for Equations of Mathematical Physics}},
   publisher={Providence, RI: American Mathematical Society (AMS)},
        date={2002},
        ISBN={0-8218-2950-5/hbk},
}

\bib{vcz}{article}{
      author={{Chepyzhov}, V.},
      author={{Vishik}, M.},
      author={{Zelik}, S.},
       title={{Regular attractors and nonautonomous perturbations of them}},
        date={2013},
        ISSN={1064-5616; 1468-4802/e},
     journal={{Sb. Math.}},
      volume={204},
      number={1},
       pages={1\ndash 42},
}

\bib{cees}{article}{
      author={{Collet}, P.},
      author={{Eckmann}, J.-P.},
      author={{Epstein}, H.},
      author={{Stubbe}, J.},
       title={{A global attracting set for the Kuramoto-Sivashinsky equation}},
        date={1993},
        ISSN={0010-3616; 1432-0916/e},
     journal={{Commun. Math. Phys.}},
      volume={152},
      number={1},
       pages={203\ndash 214},
}

\bib{col3}{article}{
      author={{Colliander}, J.},
      author={{Keel}, M.},
      author={{Staffilani}, G.},
      author={{Takaoka}, H.},
      author={{Tao}, T.},
       title={{Sharp global well-posedness for KdV and modified KdV on $\mathbb
  R$ and $\mathbb T$}},
        date={2003},
     journal={{J. Amer. Math. Soc.}},
      volume={16},
      number={3},
       pages={705\ndash 749},
}

\bib{col2}{article}{
      author={{Colliander}, J.},
      author={{Keel}, M.},
      author={{Staffilani}, G.},
      author={{Takaoka}, H.},
      author={{Tao}, T.},
       title={{Symplectic nonsqueezing of the Korteweg-de Vries flow}},
        date={2005},
     journal={{Acta Math.}},
      volume={195},
       pages={197\ndash 252},
}

\bib{col1}{article}{
      author={{Colliander}, J.},
      author={{Keel}, M.},
      author={{Staffilani}, G.},
      author={{Takaoka}, H.},
      author={{Tao}, T.},
       title={{Transfer of energy to high frequencies in the cubic defocusing
  nonlinear Schr\"odinger equation}},
        date={2010},
     journal={{Invent. Math.}},
      volume={181},
      number={1},
       pages={39\ndash 113},
}

\bib{DT}{article}{
      author={{Doelman}, A.},
      author={{Titi}, E.S.},
       title={{Regularity of solutions and the convergence of the Galerkin
  method in the Ginzburg-Landau equation}},
        date={1993},
        ISSN={0163-0563; 1532-2467/e},
     journal={{Numer. Funct. Anal. Optim.}},
      volume={14},
      number={3-4},
       pages={299\ndash 321},
}

\bib{efz}{article}{
      author={{Efendiev}, M.},
      author={{Zelik}, S.},
       title={{The regular attractor for the reaction-diffusion system with a
  nonlinearity rapidly oscillating in time and its averaging}},
        date={2003},
        ISSN={1079-9389},
     journal={{Adv. Differ. Equ.}},
      volume={8},
      number={6},
       pages={673\ndash 732},
}

\bib{erc}{article}{
      author={{Ercolani}, N.},
      author={{McLaughlin}, D.},
      author={{Roitner}, H.},
       title={{Attractors and transients for a perturbed periodic KdV equation:
  a nonlinear spectral analysis}},
        date={1993},
     journal={{J. Nonlinear Sci.}},
      volume={3},
       pages={477\ndash 539},
}

\bib{g94}{article}{
      author={{Goodman}, J.},
       title={{Stability of the Kuramoto-Sivashinsky and related systems}},
        date={1994},
        ISSN={0010-3640; 1097-0312/e},
     journal={{Commun. Pure Appl. Math.}},
      volume={47},
      number={3},
       pages={293\ndash 306},
}

\bib{GuoSimonTiti}{article}{
      author={{Guo}, Y.},
      author={{Simon}, K.},
      author={{Titi}, E.S.},
       title={{On a nonlinear system of coupled KdV equations}},
        date={2015},
     journal={{Commun. Math. Sci.}},
      volume={13},
      number={5},
       pages={1261\ndash 1288},
}

\bib{ha}{article}{
      author={{Hale}, J.},
       title={{Asymptotic behavior of gradient dissipative systems}},
        date={1987},
     journal={{Dynamics of infinite dimensional systems, Proc. NATO Adv.}},
      volume={37},
       pages={123\ndash 128},
}

\bib{he}{book}{
      author={{Henry}, D.},
       title={{Geometric Theory of Semilinear Parabolic Equations. Lecture
  Notes in Mathematics. Vol. 840}},
   publisher={{Springer-Verlag. Berlin-Heidelberg-New York}},
        date={1981},
}

\bib{ZKO}{book}{
      author={{Kozlov}, S.},
      author={{Olejnik}, O.},
      author={{Zhikov}, V.},
       title={{Homogenization of Differential Operators and Integral
  Functionals}},
   publisher={Berlin: Springer-Verlag},
        date={1994},
        ISBN={3-540-54809-2/hbk},
}

\bib{Kuksin}{book}{
      author={{Kuksin}, S.},
       title={{Nearly Integrable Infinite-Dimensional Hamiltonian Systems}},
   publisher={Berlin: Springer-Verlag},
        date={1993},
        ISBN={3-540-57161-2/pbk},
}

\bib{lg}{book}{
      author={{Levitan}, B.},
      author={{Zhikov}, V.},
       title={{Almost Periodic Functions and Differential Equations}},
   publisher={{Cambridge etc.: Cambridge University Press}},
        date={1982},
}

\bib{LT}{article}{
      author={{Liu}, H.},
      author={{Tadmor}, E.},
       title={{Rotation prevents finite-time breakdown}},
        date={2004},
     journal={{ Phys. D}},
      volume={188},
      number={304},
       pages={262–\ndash 276},
}

\bib{LM}{book}{
      author={{Lochak}, P.},
      author={{Meunier}, C.},
       title={{ Multiphase Averaging for Classical Systems with Applications to
  Adia\-ba\-tic Theo\-rems}},
   publisher={Applied Mathematical Sciences 72, Springer-Verlag New York},
        date={1988},
}

\bib{mit}{book}{
      author={{Mitroploskii}, Yu.},
       title={{The Averaging Method in Nonlinear Mechanics}},
   publisher={Kiev: Naukova Dumka},
        date={1971},
}

\bib{NST}{article}{
      author={{Nikolaenko}, B.},
      author={{Scheurer}, B.},
      author={{Temam}, R.},
       title={{Some global dynamical properties of the Kuramoto-Sivashinsky
  equations: nonlinear stability and attractors}},
        date={1985},
     journal={{ Phys. D}},
      volume={16},
      number={2},
       pages={155\ndash 183},
}

\bib{Oga}{article}{
      author={{Ogawa}, T.},
       title={{Travelling wave solutions to a perturbed Korteweg-de Vries
  equation}},
        date={1994},
        ISSN={0018-2079},
     journal={{Hiroshima Math. J.}},
      volume={24},
      number={2},
       pages={401\ndash 422},
}

\bib{ot}{article}{
      author={{Otto}, F.},
       title={{Optimal bounds on the Kuramoto-Sivashinsky equation.}},
        date={2009},
        ISSN={0022-1236},
     journal={{J. Funct. Anal.}},
      volume={257},
      number={7},
       pages={2188\ndash 2245},
}

\bib{OTZ}{article}{
      author={{Ovsyannikov}, I.},
      author={{Turaev}, D.},
      author={{Zelik}, S.},
       title={{Bifurcation to chaos in the complex Ginzburg-Landau equation
  with large 3rd order dispersion}},
        date={2015},
     journal={Modeling and Analysis of Information Systems},
      volume={22},
      number={3},
       pages={327\ndash 366},
}

\bib{PS}{book}{
      author={{Pavliotis}, G.},
      author={{Stuart}, A.},
       title={{Multiscale Methods. Averaging and Homogenization}},
   publisher={Texts in Applied Mathematics, 53. Springer, New York},
        date={2008},
}

\bib{poi}{book}{
      author={{Poincar\'e}, H.},
       title={{Les M\'ethodes Nouvelles de la M\'ecanique C\'eleste. Tome I.
  Solutions P\'eriodiques. Non-existence des Int\'egrales Uniformes. Solutions
  Asymptotiques}},
   publisher={{Paris: Gauthier-Villars et Fils}},
        date={1892},
}

\bib{SVM}{book}{
      author={{Sanders}, J.},
      author={{Verhulst}, F.},
      author={{Murdock}, J.},
       title={{Averaging Methods in Nonlinear Dynamical Systems}},
     edition={2nd ed.},
   publisher={New York, NY: Springer},
        date={2007},
        ISBN={978-0-387-48916-2/hbk; 978-0-387-48918-6/ebook},
}

\bib{tt}{article}{
      author={{Tadmor}, E.},
      author={{Tao}, T.},
       title={{Velocity averaging, kinetic formulations, and regularizing
  effects in quasi-linear PDEs}},
        date={2007},
     journal={{Comm. Pure Appl. Math.}},
      volume={60},
      number={10},
       pages={1488\ndash 1521},
}

\bib{tem}{book}{
      author={{Temam}, R.},
       title={{Infinite-dimensional Dynamical Systems in Mechanics and
  Physics}},
     edition={2nd ed.},
   publisher={New York, NY: Springer},
        date={1997},
        ISBN={0-387-94866-X/hbk},
}

\bib{W1}{article}{
      author={{Wirosoetisno}, D.},
       title={{Navier-Stokes equations on a rapidly rotating sphere}},
        date={2015},
     journal={{Discrete Contin. Dyn. Syst. Ser. B}},
      volume={20},
      number={4},
       pages={1251\ndash 1259},
}

\bib{zel}{article}{
      author={{Zelik}, S.},
       title={{Global averaging and parametric resonances in damped semilinear
  wave equations}},
        date={2006},
        ISSN={0308-2105; 1473-7124/e},
     journal={{Proc. R. Soc. Edinb., Sect. A, Math.}},
      volume={136},
      number={5},
       pages={1053\ndash 1097},
}

\end{biblist}
\end{bibdiv}

\end{document}